\documentclass[a4paper,11pt, reqno]{amsart} 
\usepackage{amssymb,amsthm,amsmath}
\usepackage{multirow}
\usepackage{enumerate}
\usepackage{ifthen}
\usepackage{graphicx} 
\usepackage{tkz-euclide}
\usepackage[colorlinks,citecolor=blue,pagebackref,hypertexnames=false]{hyperref}

 \numberwithin{equation}{section}

\usepackage{amsfonts}
\usepackage{amscd}
\usepackage{amssymb}
\usepackage{enumerate}
\usepackage{graphicx}
\allowdisplaybreaks
\usepackage{mathabx}
\usepackage{color}
\usepackage{amsbsy}
\usepackage{graphicx}
\usepackage{amsthm}
\usepackage{amsmath}
\usepackage{amsxtra}
\usepackage{mathrsfs}
\usepackage{bbm}
\usepackage{dsfont}

\usepackage{ifthen}
\usepackage{xcolor,colortbl}

\newtheorem{theorem}{Theorem}[section]
\newtheorem{cor}[theorem]{Corollary}
\newtheorem{lem}[theorem]{Lemma}

\theoremstyle{definition}

\theoremstyle{remark}
\newtheorem{rem}[theorem]{Remark}
\numberwithin{equation}{section}
\definecolor{red}{rgb}{1.0, 0.0, 0.0}
\newcommand{\Bea}{\begin{eqnarray*}}
	\newcommand{\Eea}{\end{eqnarray*}}
\newcommand{\Be} {\begin{equation*}}
	\newcommand{\Ee} {\end{equation*}}
\newcommand{\be} {\begin{equation}}
	\newcommand{\ee} {\end{equation}}
\newcommand{\bea} {\begin{eqnarray}}
	\newcommand{\eea} {\end{eqnarray}}

\newcommand{\R}{\mathbb R}

\newcommand{\K}{\kappa }

{\qed\bigskip}

\newcounter{alphabet}

\ifx\undefined\bysame
\newcommand{\bysame}{\leavevmode\hbox to3em{\hrulefill}\,}
\fi
\usetikzlibrary{patterns}

\title[Strichartz estimates for higher order Schr\"odinger equations]
{Strichartz estimates for higher order Schr\"odinger equations with  Partial regular initial data}

\author{Vishvesh Kumar} 

\address{Vishvesh Kumar  \endgraf Department of Mathematics: Analysis, Logic and Discrete Mathematics	\endgraf Ghent University \endgraf Krijgslaan 281, Building S8,	B 9000 Ghent, Belgium.} \email{Vishvesh.Kumar@UGent.be and vishveshmishra@gmail.com}

\author{Shyam Swarup Mondal} \address{Shyam Swarup Mondal    \endgraf Stat-Math Unit\endgraf Indian Statistical Institute  Kolkata \endgraf BT Road,  Baranagar, Kolkata  700108, India } \email{mondalshyam055@gmail.com}

\author{Iswarya Sitiraju} \address{Iswarya Sitiraju  \endgraf Department of Mathematics \endgraf Indian Institute of Technology Bombay \endgraf  Powai, Mumbai, Maharashtra 40007,6 India } \email{isitir1@lsu.edu and iswaryasitiraju1@gmail.com}

\author{Manli Song}

\address{Manli Song \endgraf School of Mathematics and Statistics, \endgraf Northwestern Polytechnical University, \endgraf Xi'an, Shaanxi 710129, China}
\email{mlsong@nwpu.edu.cn}

 \date{\today}
\begin{document}
	\allowdisplaybreaks
	
	\begin{abstract}
		In this paper, we establish refined Strichartz estimates for higher-order Schr\"odinger equations with initial data exhibiting partial regularity. By partial regularity, we mean that the initial data are not required to have full Sobolev regularity but only regularity with respect to a subset of the spatial variables. As an application of these estimates, we investigate the well-posedness of nonlinear Schr\"odinger equations with power-type nonlinearities.

In addition, we extend our analysis to the Dunkl–Schr\"odinger equations under partial regularity, defined with respect to two distinct root systems. This extension poses significant challenges, mainly due to the lack of a suitable stationary phase method in the Dunkl setting. To overcome this difficulty, we develop a new result that provides an adaptation of the stationary phase method to the framework of Dunkl analysis.

	\end{abstract}
	\maketitle
	\tableofcontents 	
	
	\section{Introduction and main results}
	This paper aims to investigate Strichartz estimates and their application for a class of dispersive equations associated with the Laplacian and the Dunkl Laplacian on $\mathbb{R}^d$ with power-type nonlinearities. Let us begin our discussion by considering a class of dispersive equations on $\mathbb{R}^d:$
\begin{align}\label{eq10}
		\left\{\begin{aligned}
			i \partial_t u(t, x)+\Phi(D) u(t, x) & =F_p(u)(t, x), \quad (t, x) \in \mathbb{R}_+ \times \mathbb{R}^{d}, \\
			u(0, x) & =f(x),
		\end{aligned}\right.
	\end{align}
	where $\Phi(D)$ denotes the Fourier (or Dunkl-Fourier) multiplier operator associated with the real-valued function (also known as multiplier) $\Phi$. The function $\Phi$ is assumed to satisfy certain specific, yet sufficiently general conditions that encompass several important dispersive equations discussed below.     Regarding the nonlinearity $F_p$, we assume that it belongs to the class $C^1$ and satisfies the growth condition with exponent $p > 1$:
	\begin{equation} \label{nonlin}
	    \left|F_p(u)\right| \lesssim|u|^p \quad \text { and } \quad|u|\left|F_p^{\prime}(u)\right| \sim\left|F_p(u)\right| .
	\end{equation}
	Typical examples of such nonlinearities include the standard forms $\pm |u|^{p-1}u$ and $\pm |u|^p$.

Our initial focus will be on equation \eqref{eq10} in the Euclidean setting, where $\Phi(D)$ denotes a Fourier multiplier on $\mathbb{R}^d$. We will revisit and extend the discussion to the Dunkl framework in a subsequent section. 	Equation \eqref{eq10} serves as a unifying framework that includes several fundamental dispersive equations, depending on the specific choice of the Fourier multiplier $\Phi$. For instance, when $\Phi(\xi) = |\xi|^2$, the solution $u$ to \eqref{eq10} corresponds to the classical Schrödinger evolution $e^{-i t \Delta} f$. If $\Phi(\xi) = |\xi|$, the equation models the linear wave propagation through the operator $ ^{i t \sqrt{-\Delta}} f$.  More generally, for $\Phi(\xi) = |\xi|^m$ with $m > 1$, the solution describes the fractional Schrödinger flow given by $e^{i t (-\Delta)^{m/2}} f$. These examples illustrate the flexibility of the formulation in capturing a broad class of dispersive dynamics through appropriate choices of $\Phi$.

	It is well known that equation \eqref{eq10} with initial data $f \in \dot{H}^s(\mathbb{R}^{d})$, when coupled with the nonlinearities $\pm |u|^{p-1}u$ or $\pm |u|^p$, exhibits scaling invariance due to the homogeneity of the symbol $\Phi$. This invariance under suitable scaling transformations reflects the balance between the linear dispersive term and the nonlinear interaction, a feature that plays a crucial role in the study of criticality and well-posedness. Indeed, we first note that if $u(t,x)$ is a solution of \eqref{eq10}, then so is	
	\begin{equation*}	
		u_{\delta}(t, x)= \delta^{\frac{m}{p-1}} u(\delta^{m} t, \delta x), \quad \delta>0,
	\end{equation*}	 
	with  rescaled initial data $f_{\delta}(x)=u_{\delta}(0, x)$. Now, we observe that 
	\begin{equation}\label{scaling}
		\| f_{\delta}\|_{\dot{H}^s(\mathbb{R}^{d})} =\delta^{\frac{m}{p-1}+s-\frac{d}{2}} \|f\|_{ \dot{H}^s(\mathbb{R}^{d})}.
	\end{equation}
	The scaling relation \eqref{scaling} plays a key role in identifying the critical exponent $p$ (or, equivalently, the scale-invariant Sobolev space $\dot{H}^{s_c}$), given by
	$$
	p=p_c(d, s)=1+\frac{2m}{d-2s}, \quad \text{alternately}\,\, s_c= \frac{d}{2}-\frac{m}{p-1}.
	$$
	
    This critical exponent marks the threshold at which the interplay between the linear dispersion and the nonlinearity changes. It distinguishes between different regimes of solution behavior: when $p > p_c(d, s)$, the problem is said to be supercritical; when $p < p_c(d, s)$, it is subcritical. The critical exponent serves as a dividing line between the existence of global-in-time Sobolev solutions (in the subcritical and critical cases) and the potential for finite-time blow-up of weak solutions, even with small initial data, in the supercritical case.
	
	The subcritical regime $p < p_c(d, s) = 1 + \frac{2m}{d - 2s}$ reveals that higher regularity of the initial data, corresponding to larger Sobolev indices $s$, permits a wider range of admissible nonlinear exponents $p$. This reflects the enhanced control over the nonlinearity afforded by smoother initial profiles. 

Our primary prototype for the symbol $\Phi$ is given by $\Phi(\xi) = |\xi|^m$ with $m > 1$. In this case, equation \eqref{eq10} reduces to the higher-order  Schr\"odinger equation, specifically the Cauchy problem:
	\begin{align}\label{higher order}
		\begin{cases}
i\partial_t u (t, x) + (-\Delta)^{m/2} u (t, x) = F_p(u)(t, x), \\
u(0, x) = f(x),
\end{cases}
	\end{align}
    for $(t, x) \in \mathbb{R}_+ \times \mathbb{R}^d.$ Here $(-\Delta)^{m/2}$ denotes the pseudo-differential operator given by 
    $$(-\Delta)^{m/2} f(x):= \int_{\mathbb{R}^d} e^{i x\xi} |\xi|^m \widehat{f}(\xi) \,d\xi$$
    with the Euclidean Fourier transform  $\widehat{f}(\xi):= \frac{1}{(2\pi)^d} \int_{\mathbb{R}^d} e^{-i x\xi} f(x) dx.$

The higher-order Schr\"odinger equation with $m > 2$ has garnered significant interest in mathematical physics. In particular, when $m = 4$, equation \eqref{higher order} arises in the study of the formation and propagation of intense laser beams in bulk media; see \cite{Ka96, KS00}.  Fractional-order cases with $1 < m < 2$ are also of considerable importance, especially in light of the development of fractional quantum mechanics by Laskin \cite{La00}, where the fractional Schr\"odinger equation was introduced as a generalization of the classical model. This formulation stems from an extension of the Feynman path integral from Brownian-like to L\'evy-like quantum mechanical paths (see \cite{La00, La02}). In this framework, it is conjectured that certain quantum phenomena may be accurately captured only within these fractional regimes. Additionally, the fractional Schr\"odinger equation appears in water wave models (see \cite{Waterwaveref}), further highlighting its interdisciplinary relevance. When $m \in [2, \infty)$, the nonlinear fractional Schr\"odinger equation  generalizes both the standard nonlinear Schr\"odinger equation ($m = 2$) and the fourth-order nonlinear Schr\"odinger equation ($m = 4$); see, for example, \cite{Paus07}.

In the classical case $m = 2$, equation \eqref{higher order} reduces to the standard nonlinear Schr\"odinger equation on $\mathbb{R}^d$, a model that arises naturally in quantum mechanics. A central aspect in the analysis of such equations is understanding the space-time behavior of solutions to the associated free linear problem, $u(t, x) = e^{  -it\Delta} f(x)$. These estimates form the backbone of well-posedness and scattering theory for nonlinear Schr\"odinger equations; see \cite{GV85, GV92, KT98, St77, Ka96} for foundational results. Among these, the Strichartz estimates play a particularly crucial role. In the classical setting, Strichartz \cite{St77} was the first to establish such estimates in the mixed-norm framework $L^q_{t,x}(\mathbb{R}^{d+1})$. Since then, they have become a fundamental tool in the study of nonlinear dispersive equations; see \cite{Bahouri-app} and references therein.

For the higher-order Schr\"odinger equation \eqref{higher order}, analogous estimates have been derived. In particular, \cite{CHL11} established that for $m>2, q, r \geq 2$, $r \neq \infty$, satisfying the scaling condition $\frac{m}{q} + \frac{d}{r} = \frac{d}{2}$, one has
\begin{align} \label{strhighor}
\left\|e^{it (-\Delta)^{\frac{m}{2}}}f\right\|_{L_t^q\left(\mathbb{R}; L_x^r\left(\mathbb{R}^{d}\right)\right)} \lesssim \|f\|_{L^2(\mathbb{R}^{d})}.
\end{align}
Such estimates are indispensable in extending the classical theory of nonlinear Schr\"odinger equations to higher-order and more general settings. Strichartz estimates for  the higher-order in different frameworks  and their application in the nonlinear PDEs can be found in \cite{KS17, KS00, Ka96, CHL11}. Based on the Strichartz estimate \eqref{strhighor}, the local well-posedness of \eqref{higher order} has been discussed in  \cite{CHL11, HS15, Dinh18, Dinh19} with initial data from Sobolev spaces $H^s(\mathbb{R}^d)$ and $\dot{H}^s(\mathbb{R}^d).$ Particularly in the subcritical case, that is $p<p_c(d, s)=1+\frac{2m}{d-2s}$, one could see that the more one assumes regularity s of initial data, the wider a possible range of $p$ in the power type nonlinearity.

The main objective of this paper is to investigate the local well-posedness of the nonlinear Schr\"odinger equation \eqref{higher order} associated with higher-order operators on $\mathbb{R}^d$, when the initial data does not possess full regularity in the classical Sobolev space $H^s(\mathbb{R}^d)$, but instead lies in a partially regular space of the form $L^2(\mathbb{R}^{d-k}; H^{s} (\mathbb{R}^k))$ for some $1 \leq k \leq d$. Recent progress in this direction was made by Koh, Lee, and Seo \cite{KLS23}, who studied local well-posedness of nonlinear Schr\"odinger and wave equations under partial regularity assumptions. Certain nonlinear Schr\"dinger equations exhibit partial off-axis variations, which naturally motivates the consideration of initial data with only partial regularity in the analytical study of such models; see \cite{AA, Ak}.
 Let us briefly explain the motivation behind this type of study.

We decompose the spatial variable $X \in \mathbb{R}^d$ as $(x, y) \in \mathbb{R}^{d-k} \times \mathbb{R}^k$, and assume that the initial data belongs to the mixed Sobolev space $L^2(\mathbb{R}^{d-k}; {\dot{H}^s(\mathbb{R}^k))}$ rather than the full space {$\dot{H}^s(\mathbb{R}^d)$}.  Under this assumption, one can still observe the same scaling behavior as in the fully regular case. More precisely, for a rescaled initial datum $f_\delta$, we have
$$\| f_{\delta}\|_{L_x^2 \dot{H}_y^s} =  \delta^{\frac{m}{p-1} + s - \frac{d}{2}} \|f\|_{L_x^2 \dot{H}_y^s},$$
where the scaling exponent $ \frac{m}{p-1} + s - \frac{d}{2}$ is identical to that in the full regularity setting \eqref{scaling}, independent of the choice of $k$.

This observation suggests that the critical exponent $p$ for well-posedness remains unchanged even under partial regularity assumptions. Consequently, it motivates the development of a well-posedness theory in lower-regularity frameworks, thereby extending the classical theory to a broader class of initial data. It is worth noting that   $L^2 (\mathbb{R}^{d-k}; H^s (\mathbb{R}^k))$ is bigger than that of $H^s(\mathbb{R}^{d})$  in the sense that $1\leq k_1 < k_2 <d$,
	$$
	H^s(\mathbb{R}^{d}) \subsetneq L^2 (\mathbb{R}^{d-k_2}; H^s (\mathbb{R}^{k_2})) \subsetneq L^2 (\mathbb{R}^{d-k_1}; H^s (\mathbb{R}^{k_1})).
	$$ We refer to \cite{KLS23} for more details on this topic.

	As a natural first step toward developing a well-posedness theory for the nonlinear Schrödinger equation \eqref{higher order} under the relaxed partial regularity assumptions discussed above, we begin by establishing Strichartz estimates for the associated free Schrödinger propagator $e^{it (-\Delta)^{\frac{m}{2}}}$. 
    
    More generally, we derive refined Strichartz estimates for the generalized dispersive propagator $e^{it \Phi(D)}$ arising from the Cauchy problem associated with a broader class of Fourier multipliers $\Phi$:

    \begin{align}\label{eq10homo}
		\left\{\begin{aligned}
			i \partial_t u(t, x)+\Phi(D) u(t, x) & =0, \quad (t, x) \in \mathbb{R} \times \mathbb{R}^{d}, \\
			u(0, x) & =f(x),
		\end{aligned}\right.
	\end{align}
    under the following  assumptions {\bf(A)} on $\Phi.$  

    \textbf{Assumptions (A):} We assume that $\Phi \in C^{\infty}\left(\mathbb{R}^{d} \backslash\{0\}\right)$ is a real-valued function satisfying the following conditions:
	\begin{enumerate}
		\item  $|\nabla \Phi(\xi, \eta)| \neq 0$ for all $(\xi, \eta) \in \mathbb{R}^{d-k} \times \mathbb{R}^{k} \backslash\{0\}$.
		\item  There is a constant $\mu \geq 1$ such that $\mu^{-1} \leq|\Phi(\xi, \eta)| \leq \mu$ for any $\xi$ with $|(\xi, \eta)|=1$.
		\item There is a constant $m \geq 1$ such that $\Phi(\lambda (\xi, \eta))=\lambda^m \Phi(\xi, \eta)$ for all $\lambda>0$ and all $(\xi, \eta) \in \mathbb{R}^{d-k} \times \mathbb{R}^{k} \backslash\{0\}$.
		\item The rank of the Hessian $H{\Phi}(\xi, \eta)\geq  M\geq k \geq 1$ for all $(\xi, \eta) \in \mathbb{R}^{d-k} \times \mathbb{R}^{k} \backslash\{0\}$.
		\item   The rank of the Hessian $H{\Phi}(\xi, \eta) \geq M-k,$ for all $(\xi, \eta) \in \mathbb{R}^{d-k} \times \mathbb{R}^{k} \backslash\{0\}$  with fixed $|\eta| \leq 2$.
	\end{enumerate}
It is worth mentioning that
\begin{itemize}
    \item Condition (3) is specifically needed for the frequency-localized estimate arising in the Littlewood–Paley decomposition.
  
    \item The remaining conditions are necessary to control the oscillatory integral $$ \int_{\mathbb{R}^{d-k} \times \mathbb{R}^k} e^{-i(x\cdot\xi+y\cdot \eta)+it\Phi(\xi, \eta)}\Psi(\xi, \eta) \;d\xi\, d\eta$$  in terms of time.
\end{itemize}

    The following theorem about refined Strichartz estimates for the generalized dispersive propagator $e^{it \Phi(D)}$ is one of the main results of this paper. 
	\begin{theorem} \label{them1.1}
		Let $d \geq 1$ and $1 \leq k \leq d$. Let $\Phi$ be a real-valued smooth function satisfying Assumptions (A).  If
		\begin{align*} 
			\frac{2}{q} \leq(M-k)\left(\frac{1}{2}-\frac{1}{r}\right)+k\left(\frac{1}{2}-\frac{1}{\widetilde{r}}\right), \quad 2 \leq \widetilde{r} \leq r<\infty, \quad 2<q \leq \infty. 
		\end{align*}
		Then, we have
		\begin{align}\label{eq1inr}
			\left\|e^{i t\Phi(D)} f\right\|_{L_t^q\left(\mathbb{R} ; L_x^r\left(\mathbb{R}^{d-k} ; L_y^{\tilde{r}}\left(\mathbb{R}^k\right)\right)\right)} \lesssim\|f\|_{\dot{H}_{x, y}^s},
		\end{align}
		under the scaling condition
		\begin{align*} 
		\frac{{{m}}}{q}=-s +(d-k)\left(\frac{1}{2}-\frac{1}{r}\right)+k\left(\frac{1}{2}-\frac{1}{\widetilde{r}}\right).
		\end{align*}
		
	\end{theorem}   

    These refined Strichartz estimates enable us to leverage varying integrability and regularity properties across different spatial directions, thereby making it possible to formulate a well-posedness theory under the assumption of only partial regularity in the initial data.

		\begin{rem} We remark that Theorem \ref{them1.1} and its proof not only generalize several known results, but also provide a unified framework that encompasses various dispersive models, including the Schr\"odinger, wave, and higher-order Schr\"odinger equations.
 In the special case where $r = \tilde{r}$, it recovers the classical Strichartz estimates for higher-order Schr\"odinger equations as established in \cite{CHL11}. When $\Phi(\xi) = |\xi|^2$ or $\Phi(\xi) = |\xi|$, the theorem aligns with the results obtained in \cite{KLS23} for the Schr\"odinger and wave equations, respectively. It is important to note that the admissibility condition in Theorem \ref{them1.1} is based on the assumption (see Assumptions {\bf (A)}) that the Hessian $H_{\Phi}(\xi)$ of $\Phi$ has rank $M \geq 1$ for all $\xi \neq 0$.
\end{rem}

Now that we have established the Strichartz estimates \eqref{eq1inr}, we turn our attention to the local-in-time well-posedness of the Cauchy problem for the nonlinear Schrödinger equation \eqref{higher order}, considering initial data with partial regularity, specifically $f \in L_x^2\left(\mathbb{R}^{d-2} ; H_y^s\left(\mathbb{R}^2\right)\right)$.
	\begin{theorem} \label{thmintowel}
		Let   $0<s \leq 1, m>2$ and $d \geq 3$.  Let   $1<p<1+\frac{2m}{ d-2s}$ and $f \in$ $L_x^2\left(\mathbb{R}^{d-2} ; H_y^s\left(\mathbb{R}^2\right)\right)$.  Further, assume that  
		$(q, r, \tilde{r})$ with $2 \leq \widetilde{r} \leq r<\infty$ and $ 2<q \leq \infty $ satisfies $$			\frac{{{m}}}{q}=(d-2)\left(\frac{1}{2}-\frac{1}{r}\right)+2\left(\frac{1}{2}-\frac{1}{\widetilde{r}}\right).$$ Then there exist  $T>0$ and a unique solution $u$ to \eqref{higher order} such that 
		$$
		u \in C_t\left([0, T] ;L_x^2\left(\mathbb{R}^{d-2} ; H_y^s\left(\mathbb{R}^2\right)\right)\right) \cap L_t^q\left([0, T] ; L_x^r \left(\mathbb{R}^{d-2}; W_y^{s, \widetilde{r}}(\mathbb{R}^{2})\right)\right).
		$$
		\end{theorem}
        \begin{rem}

        The local well-posedness theory for equation \eqref{higher order} in the Schr\"odinger case ($m = 2$) and the wave case ($m = 1$) with partially regular initial data has already been established in the literature; see \cite{KLS23}. Therefore, in Theorem \ref{thmintowel}, we restrict our attention to the higher-order case, namely $m > 2$. Nevertheless, we note that the proof of Theorem \ref{thmintowel} remains valid under the broader assumption $m \geq 2$.
        \end{rem}

    	Now, we turn our attention toward the Strichartz estimates of the Dunkl-Schr\"odinger equations and well-posedness theory for nonlinear Dunkl-Schr\"odinger equations. Over the last four decades, an alternative to classical Fourier analysis, known as \textit{(rational) Dunkl theory}, has been developed in Euclidean spaces. This framework builds upon the structure of a root system $\mathcal{R}$, its associated reflection group (Coxeter group) \( G \), and a multiplicity function \( \K: \mathbb{R} \rightarrow \mathbb{C} \).
 Let $h^2_\K$ and $E_\K$ denote the weight function and Dunkl kernel associated with the root system $\mathcal{R}$, respectively.  Let $N:=d+2\gamma$ denote the homogeneous dimension.  We refer to Section \ref{sec5} for detailed definitions and unexplained notations.

 For $\xi \in \mathbb{C}^{n}$ and for a positive multiplicity function $\kappa,$ Dunkl \cite{dun1991} in 1989 introduced a family of first-order differential-difference operators $T_{\xi}:= T_{\xi}(\kappa)$  by
	\begin{align}\label{dunklinro}
		T_{\xi}(\kappa) f(x):=\partial_{\xi} f(x)+\sum_{\alpha \in \mathcal{R}^{+}} \kappa_{\alpha }\langle \alpha, \xi\rangle \frac{f(x)-f\left(r_{\alpha} x\right)}{\langle \alpha, x\rangle}, \quad f \in C^{1}(\mathbb{R}^{d}),
	\end{align}
	where $\partial_{\xi}$ denotes the directional derivative corresponding to $\xi$. These differential-difference operators, now known as Dunkl operators, extend the classical notion of partial derivatives.  Dunkl operators have found numerous applications across mathematics and mathematical physics. For instance, they play a significant role in the integrability of quantum many-body systems of the Calogero–Moser–Sutherland type and in the analysis of certain probabilistic processes \cite{Van, Hakimi}. Moreover, they have profoundly influenced the theory of special functions and orthogonal polynomials in both one and several variables; see \cite{Dunkl book}.

Let $\{\xi_1, \xi_2, \cdots, \xi_n\}$ be an orthonormal basis of $(\mathbb{R}^{d} , \langle \cdot, \cdot\rangle )$. Then  the Dunkl Laplacian operator $\Delta_\kappa $ is defined as $$\Delta_\kappa:=\sum_{j=1}^nT^2_{\xi_j}(\kappa).$$
  Now, for $m >0 ,$ we consider the   nonlinear Dunkl-Schr\"odinger   equations
	\begin{equation}\label{DNLS1}
		\begin{cases}
			i\partial_t u + (-\Delta_\kappa)^{m/2} u= F_p(u), \\
			u(0, x)=f(x),
		\end{cases}
	\end{equation}
	where the nonlinearity $F_p \in C^1$ with $p>1$ satisfies
	\begin{equation}\label{as1}
		|F_p(u)|\lesssim |u|^p \quad \text{and} \quad |u||F'_p(u)|\sim |F_p(u)|. 
	\end{equation}
     
    Here $(-\Delta_\kappa)^{m/2}$ denotes the Dunkl-Fourier multiplier operator with symbol $|\xi|^m$ given by 
    $$(-\Delta_\kappa)^{m/2} f(x):= \int_{\mathbb{R}^d}  |\xi|^m \mathcal{F}_\kappa{f}(\xi) E_\kappa (i x, \xi) h_\kappa(\xi) d\xi $$
    with the Dunkl Fourier transform  $$\mathcal{F}_\kappa{f}(\xi):= \frac{1}{c_{\kappa}}\int_{\mathbb{R}^d}  f(x) E_\kappa (x,-i \xi) h_\kappa(x)dx.$$

	Our aim is to investigate Strichartz estimates and local well-posedness for the Dunkl–Schr\"odinger equation \eqref{DNLS} within the framework of partial regularity. Since Dunkl analysis is intrinsically connected to root systems, it is essential to first establish a suitable setting for formulating partial regularity on $\mathbb{R}^d$ in this context.
       Let $\mathcal{R}_1$ and $\mathcal{R}_2$ be root systems in $\mathbb{R}^{d-k}$ and $\mathbb{R}^k$, respectively, with associated reflection groups $G_1$ and $G_2$, where $1 \leq k \leq d$.
  Then,
	$$
	\mathcal{R}:=\mathcal{R}_1 \times(0)_{k} \cup(0)_{d-k} \times \mathcal{R}_2 ,
	$$
	where $(0)_j=(0,0, \ldots, 0) \in \mathbb{R}^j$, is a root system on $\mathbb{R}^{d}$ with reflection group $G=G_1 \times G_2$. Let $\kappa_1 $ and $\kappa_2 $   be  multiplicity functions for $\mathcal{R}_1$ and   $\mathcal{R}_2$ respectively. Then we  define the multiplicity function $\kappa: \mathcal{R} \rightarrow \mathbb{Z}_{\geq 0}$ the root system $\mathcal{R}$ is given by $$
	\begin{array}{ll}
		\kappa(\alpha, 0)=\kappa_1(\alpha) & \text { for all } \alpha \in \mathcal{R}_1 \\
		\kappa(0, \beta)=\kappa_2(\beta) & \text { for all } \beta \in \mathcal{R}_2.
	\end{array}
	$$ In fact, for any $(x, y)\in \mathbb{R}^{d}$, let $h_{\kappa_1}(x)$ and $h_{\kappa_2}(y)$ be the
	weight functions on $\mathcal{R}_1$ and $\mathcal{R}_2$ with homogeneous degree $2\gamma_1$ and  $2\gamma_2$.  Then $h_{\kappa}(x, y)=h_{\kappa_1}(x) h_{\kappa_2}(y)$ is the weight on $\mathcal{R}$  with homogeneous degree $2\gamma=2\gamma_1+2\gamma_2$.  From now on, we will fix this root system and multiplicity function. In this case, the homogeneous dimension becomes $N:=d+2\gamma_1+2\gamma_2=d+2\gamma.$

	Note that,  if $u(t,x)$ is a solution of \eqref{DNLS1}, then so is
	\begin{equation*}
		u_{\delta}(t, x)= \delta^{\frac{m}{p-1}} u(\delta^{m} t, \delta x), \quad \delta>0,
	\end{equation*}
    thanks to the scaling invariance of $\Delta_\kappa.$
	In addition, the Sobolev norm of the rescaled initial data $f_{\delta}(x)=u_{\delta}(0, x)$ is given in terms of the original $f$ as 
	\begin{equation*}
		\| f_{\delta}\|_{\dot{H}_\kappa^s(\mathbb{R}^{d})} =\delta^{\frac{m}{p-1}+s-\frac{d+2\gamma}{2}} \|f\|_{ \dot{H}_{\kappa }^s(\mathbb{R}^{d})},
	\end{equation*}
    where the homogeneous Dunkl-Sobolev space $\dot{H}_\kappa^{s}\left(\mathbb{R}^{d}\right)$ is the set of tempered distributions $u \in \mathcal{S}'\left(\mathbb{R}^{d}\right)$ such that $\mathcal{F}^{-1}_\kappa\left(|\cdot|^s \mathcal{F}_\kappa f\right)\in L_\kappa^2\left(\mathbb{R}^{d}\right)$ and is equipped with the norm 
		\begin{equation*}
			\|u\|_{\dot{H}_\kappa^{s}\left(\mathbb{R}^{d}\right)}=\left\|\mathcal{F}^{-1}_\kappa\left(|\cdot|^s \mathcal{F}_\kappa f\right)\right\|_{L_\kappa^2\left(\mathbb{R}^{d}\right)}.
		\end{equation*}
This scaling suggests the critical exponent $p$ for \eqref{DNLS1}  given by
	\begin{equation*}
	    p=p_c(d, s)=1+\frac{2m}{d+2\gamma-2s}, \quad \text{alternately}\,\, s_c= \frac{d+2\gamma}{2}-\frac{m}{p-1}.
	\end{equation*}
    In the classical case of the nonlinear Dunkl–Schrödinger equation, corresponding to $m = 2$, Strichartz estimates for the free Dunkl–Schrödinger propagator $e^{it\Delta_\kappa}$ have been established in \cite{Mejjaoli2009, shyam}. As a consequence, the local well-posedness of the nonlinear Dunkl–Schrödinger equation with initial data in $L^2_\kappa(\mathbb{R}^d)$ has been proved for the range $1 < p < \frac{4}{d + 2\gamma}$; see \cite{Mejjaoli2009}.

	On the other hand 
	$$\| f_{\delta}\|_{L_{\K_1,x}^2 \dot{H}_{\kappa_2 ,y}^s} =\delta^{\frac{m}{p-1}+s-\frac{d+2(\gamma_1+\gamma_2)}{2}} \|f\|_{L_{\K_1,x}^2 \dot{H}_{\kappa_2, y}^s}=\delta^{\frac{m}{p-1}+s-\frac{N}{2}} \|f\|_{L_{\K_1,x}^2 \dot{H}_{\kappa_2, y}^s},$$
regardless of the value of $k$. This shows that the same range of $p$ is can be guarantee even if we impose regularity only partially.

    An essential tool to prove well-posedness theory is the Strichartz estimates. Strichartz estimates for the Dunkl-Schr\"{o}dinger propagator $e^{it\Delta_\kappa}$ can be seen   \cite{Mejjaoli2009, shyam} and references therein. Due to the importance as described above,  an interesting and viable problem
	is to study the Strichartz estimate for the Dunkl-Schr\"odinger equation in the partial regularity framework: 
    \begin{equation*}
		\begin{cases}
			i\partial_t u(t, x) + (-\Delta_\kappa)^{m/2} u(t, x)= 0, \quad (t, x) \in \mathbb{R}_+ \times \mathbb{R}^d, \\
			u(0, x)=f(x),
		\end{cases}
	\end{equation*} for $m>0.$
    In this direction, the following Strichartz estimate for the free Dunkl-Schr\"odinger propagation $u(\cdot, \cdot ):=e^{i t(-\Delta_\kappa)^{\frac{m}{2}}} f(\cdot)$ is one of the main results. 
	\begin{theorem} \label{Dunklshrodintro}
		Let $d\geq 1,$ $1 \leq k \leq d,$ $  2 \leq \widetilde{r} \leq r<\infty , m\in (0, \infty)\backslash \{1\} $  and $ 2<q \leq \infty$. If $ q, r, \tilde{r}$ satisfy 
		\begin{align*}
			\frac{2}{q} \leq(d+2\gamma_1-k)\left(\frac{1}{2}-\frac{1}{r}\right)+(k+2\gamma_2)\left(\frac{1}{2}-\frac{1}{\widetilde{r}}\right). 
		\end{align*}
		Then we have
		\begin{align*}
			\left\|e^{i t(-\Delta_\kappa)^{\frac{m}{2}}} f\right\|_{L_t^q\left(\mathbb{R} ; L_{{ \kappa_1},x}^r\left(\mathbb{R}^{d-k} ; L_{{ \kappa_2},y}^{\tilde{r}}\left(\mathbb{R}^k\right)\right)\right)} \lesssim\|f\|_{\dot{H}_{\kappa}^s}^2
		\end{align*}
		under the scaling condition
		\begin{align*} 
			\frac{{{m}}}{q}=-s +(d+2\gamma_1-k)\left(\frac{1}{2}-\frac{1}{r}\right)+(k+2\gamma_2)\left(\frac{1}{2}-\frac{1}{\widetilde{r}}\right).
		\end{align*}
	\end{theorem}   
	We establish this result in Subsection \ref{sec6} by employing the Littlewood–Paley theory in the Dunkl framework, together with the following intresting result, which serves as the Dunkl analogue of the classical lemma from the stationary phase method. 

    \begin{lem}\label{fractionalinr}
			Let  $d \geq 1,  \psi$ be any smooth function on $\mathbb{R}$ supported in $[\frac{1}{2},2]$ and let $\phi: \mathbb{R}^+ \to \mathbb{R}$ be a smooth function.
			\begin{itemize}
				\item If $\phi'(r)\neq0,$ $\forall r\in [\frac{1}{2},2]$, for $(x, t) \in \mathbb{R}^{d+1}$, we have
				$$
				\left|\int_{\mathbb{R}^{d}} E_{\kappa}(ix, \xi) e^{-i t\phi(|\xi|)} \psi(|\xi|) h_\kappa(\xi)  d \xi\right| \lesssim_{\phi,\psi}
				(1+|(x, t)|)^{-\frac{N-1}{2}}.
				$$
				\item If $\phi'(r),\phi''(r)\neq0,$ $\forall r\in [\frac{1}{2},2]$, for $(x, t) \in \mathbb{R}^{d+1}$, we have
				$$
				\left|\int_{\mathbb{R}^{d}} E_{\kappa}(ix, \xi) e^{-i t\phi(|\xi|)} \psi(|\xi|) h_\kappa(\xi)  d \xi\right| \lesssim_{\phi,\psi}
				(1+|(x, t)|)^{-\frac{N}{2}}.
				$$
			\end{itemize}
	\end{lem}

    Our last result of this paper concerns the Strichartz estimates for the Dunkl wave propagator, which is a solution to the following Dunkl-wave equation
    \begin{equation*}
		\begin{cases}
			i\partial_t u(t, x) + (-\Delta_\kappa)^{1/2} u(t, x)= 0, \quad (t, x) \in \mathbb{R}_+ \times \mathbb{R}^d, \\
			u(0, x)=f(x),
		\end{cases}
	\end{equation*}
    with  initial data $f$ in $\dot{H}_{\kappa}^s$.

    \begin{theorem}\label{Dunklwaveintro}
		Let $N_1=d+2\gamma_1 \geq 2,$ $1 \leq k \leq N_1-1,$ $ 2 \leq \widetilde{r} \leq r<\infty $ and $ 2<q \leq \infty$. If $ q, r, \tilde{r}$ satisfy 
		\begin{align*}
			\frac{2}{q} \leq(d+2\gamma_1-k-1)\left(\frac{1}{2}-\frac{1}{r}\right)+(k+2\gamma_2)\left(\frac{1}{2}-\frac{1}{\widetilde{r}}\right). 
		\end{align*}
		Then we have
		\begin{align*}
			\left\|e^{i t(-\Delta_\kappa)^{\frac{1}{2}}} f\right\|_{L_t^q\left(\mathbb{R} ; L_{{\kappa_1},x}^r\left(\mathbb{R}^{d-k} ; L_{{ \kappa_2},y}^{\tilde{r}}\left(\mathbb{R}^k\right)\right)\right)} \lesssim\|f\|_{\dot{H}_{\kappa}^s}
		\end{align*}
		under the scaling condition
		\begin{align*}
			\frac{{{1}}}{q}=-s +(d+2\gamma_1-k)\left(\frac{1}{2}-\frac{1}{r}\right)+(k+2\gamma_2)\left(\frac{1}{2}-\frac{1}{\widetilde{r}}\right).
		\end{align*}
	\end{theorem}

    \begin{rem} 
As an application of Theorems \ref{Dunklshrodintro} and \ref{Dunklwaveintro}, it is natural to investigate the well-posedness of the nonlinear higher-order Dunkl–Schrödinger equation (i.e., \eqref{DNLS1} with $m>1$) and the Dunkl–wave equation (i.e., \eqref{DNLS1} with $m=1$) with nonlinearities of the type \eqref{as1}, for partially regular initial data. This is in analogy with Theorem \ref{thmintowel} and with the Euclidean case studied in \cite{KLS23}. A crucial ingredient in establishing well-posedness in the partial regularity framework is the fractional generalized chain rule \cite{CW91, KLS23}. Since such a tool is not yet available in Dunkl analysis, an interesting problem in its own right, we plan to address the well-posedness of these Dunkl dispersive equations in this framework, incorporating the fractional chain rule, in a forthcoming paper. It is worth mentioning here that a first step toward a generalized Dunkl chain rule has been initiated in \cite{BM25} by establishing fractional Leibniz rules for the Dunkl operator.
 \end{rem}
	
	Apart from the introduction, the paper is organized as follows: In Section \ref{sec2}, we recall the Littlewood-Paley operator and some of its important properties. Section \ref{sec3} is devoted to proving Strichartz estimates for the class of dispersive semigroups $e^{it\Phi(D)}$ with regular initial data. As an application of the Strichartz estimates, we investigate the local well-posedness result for the nonlinear higher order Schr\"odinger equation (\ref{higher order}) with partially regular initial data in Section \ref{sec4}. In Section \ref{sec5}, first, we recall harmonic analysis associated with the Dunkl operator and prove  Strichartz estimates for the class of dispersive semigroups $e^{it(-\Delta_\kappa)^{\frac{m}{2}}}$ in the partial regularity framework. 
	
    In this paper, we denote by \( C \) a generic positive constant, which may vary from line to line. The notation \( A \lesssim B \) means that \( A \leq C B \) for some constant \( C > 0 \), while \( A \sim B \) indicates that \( C^{-1} B \leq A \leq C B \), with \( C > 0 \) again denoting an unspecified constant.

    \section{Preliminaries}\label{sec2}
    This section is devoted to presenting basic facts, function spaces, and preliminary tools that will be used throughout the paper.

As discussed in the introduction, we denote \( X = (x, y) \in \mathbb{R}^{d-k} \times \mathbb{R}^k \) for \( 1 \leq k \leq d \), with the understanding that when \( k = d \), the \( y \)-variable does not appear. We make use of mixed Lebesgue spaces of the form \( L^p_x(\mathbb{R}^{d-k}, L^q_y(\mathbb{R}^k)) \), which we also denote by \( L^p_xL^q_y(\mathbb{R}^{d-k} \times \mathbb{R}^k) \) or simply \( L^p_xL^q_y \). These spaces are equipped with the norm:
$$ \|f\|_{L^p_xL^q_y} := \left( \int_{\mathbb{R}^{d-k}} \left( \int_{\mathbb{R}^k}  |f(x, y)|^q\, dy \right)^{\frac{p}{q}} dx \right)^{\frac{1}{p}}.$$

As usual, we define the Fourier transform of a function $f=f(x, y)$ as 
$$(\mathcal{F}f)(\xi, \eta):= \widehat{f}(\xi, \eta):= \frac{1}{(2\pi)^d} \int_{\mathbb{R}^d}f(x, y) e^{-i (x. \xi+y.\eta)}\,dx\,dy.$$
On the other hand, the partial Fourier transform with respect to the $y$-variable is defined as 
$$(\mathcal{F}_yf)(x, \eta) :=\frac{1}{(2\pi)^k} \int_{\mathbb{R}^k}f(x, y) e^{-i y.\eta}\,dy. $$
In a similar fashion, the partial Fourier transform with respect to $x$-variable will be denoted by $\mathcal{F}_x.$
Next, we will introduce the mixed Sobolev-type spaces $L^p(\mathbb{R}_x^{d-k}, H^s(\mathbb{R}^{k}_y))$ of order $s \in \mathbb{R}$ via the following norm
$$\|f\|_{L^p_xH^s_y}:= \|(1-\Delta_y)^{\frac{1}{2}} f\|_{L^p_xL^2_y}= \|(1+|\eta|^2)^{\frac{s}{2}} f\|_{L^p_xL^2_y}. $$
    Furthermore, for any interval $I \subset \mathbb{R},$ we will make use of the mixed space-time space $L^q_t(I, L^p_x(\mathbb{R}^{d-k}, H^s_y(\mathbb{R}^{k})))$ which will be briefly denotes by $ L^q_t(I)L^p_xH_y^s$ or simply $L^q_tL^p_xH_y^s.$ These spaces are eqquiped with the norm 
$$\|F\|_{L^q_tL^p_xH_y^s}:= \left( \int_I \|F(t)\|^q_{L^p_xH^s_y}\, dt \right)^{\frac{1}{q}}.$$

    For suitable functions $f$ and $g$ defined on $\R^{d-k} \times \R^k,$ we denote by $f*_y g$ the convolution of $f$ and $g$ with respect to $y$-variable defined as 
$$f*_yg (x, y):= \int_{\R^k} f(x, y-y') g(x, y')\, dy'. $$

	Let $\psi: \mathbb{R}^{d} \rightarrow[0,1]$ be a radial smooth cut-off function supported in $$\{(\xi, \eta) \in\left.\mathbb{R}^{d-k} \times \mathbb{R}^k: \frac{1}{2} \leq|(\xi, \eta)| \leq 2\right\}$$ such that
	$$
	\sum_{j \in \mathbb{Z}} \psi\left(2^{-j} \xi, 2^{-j} \eta\right)=1.
	$$
	
 	For $j \in \mathbb{Z}$, the Littlewood-Paley operator $P_j$ is defined as follows:
	\begin{equation} \label{lpope}
	    \widehat{P_j f}(\xi, \eta)=\psi_j(\xi, \eta) \hat{f}(\xi, \eta),
	\end{equation}
	where $\psi_j(\xi, \eta):=\psi\left(2^{-j} \xi, 2^{-j} \eta\right)$  and its support is the set  $\left\{(\xi, \eta): 2^{j-1} \leq|(\xi, \eta)| \leq 2^{j+1}\right\}$.The Littlewood-Paley theorem  on mixed Lebesgue spaces $L_x^r L_y^{\widetilde{r}}\left(\mathbb{R}^{d-k} \times \mathbb{R}^k\right)$; for $1<r, \widetilde{r}<\infty$,
	is given by 	\begin{align}\label{norm}
	\|f\|_{L_x^r L_y^{\tilde{r}}} \lesssim\left\| \left(\sum_{j \in \mathbb{Z}} \left|P_j f\right|^2\right)^{1 / 2}\right\|_{L_x^r L_y^{\tilde{r}}} .
	\end{align}
	The    Littlewood-Paley operators    $P_j$ satisfy the following important properties. 
	\begin{itemize}
		\item Identity relation: $f= \displaystyle \sum_{j \in \mathbb{Z}} P_j f.$
		\item  Quasi-orthogonal relations: $P_jP_k=0$ if $|j-k|>1.$
		\item There exist constant $C>0$ such that $$C^{-1} \sum_{j \in \mathbb{Z}}\left\| P_j f\right\|_{L^2\left(\mathbb{R}^{d}\right)}^2 \leq\|f\|_{L^2\left(\mathbb{R}^{d}\right)}^2 \leq C \sum_{j \in \mathbb{Z}}\left\|P_j f\right\|_{L^2\left(\mathbb{R}^{d}\right)}^2.$$
		\item  Sobolev norm: $$\|f\|_{H^s\left(\mathbb{R}^{d}\right)}\sim   \left\|\{2^{js}\left\| P_j f\right\|_{L^2\left(\mathbb{R}^{d}\right)}\}_{j\in\mathbb{Z}}\right \|_{\ell^2(\mathbb{Z})} .$$
		Moreover, there exist constant $C>0$ such that
		$$C^{-1} \sum_{j \in \mathbb{Z}} 2^{2js}\left\| P_j f\right\|_{L^2\left(\mathbb{R}^{d}\right)}^2 \leq\|f\|_{H^s\left(\mathbb{R}^{d}\right)}^2 \leq C \sum_{j \in \mathbb{Z}}2^{2js}\left\|P_j f\right\|_{L^2\left(\mathbb{R}^{d}\right)}^2.$$
	\end{itemize}
	To establish one of the main results of this paper, we make use of the following classical lemma from the stationary phase method (see \cite[VIII, Section 5, B]{S2}). For convenience, we denote the Hessian matrix  $\left(\frac{\partial^2}{\partial \xi_i \partial \xi_j}\right)$ by $H$.
	\begin{lem}\label{eq8}
		 Let $\psi$ be a compactly supported smooth function on $\mathbb{R}^{d}$ and let $\phi$ be  a smooth function satisfying rank $H \phi \geq M$ on the support of $\psi$. Then, for $(x, t) \in \mathbb{R}^{d+1},$ we have 
		$$
		\left|\int_{\mathbb{R}^d} e^{i(x, t) \cdot(\xi, \phi(\xi))} \psi(\xi) d \xi\right| \leq C(1+|(x, t)|)^{-M/ 2} .
		$$
		\end{lem}
        
		\section{Strichartz estimates for the dispersive semigroup $e^{it\Phi(D)}$} \label{sec3} 
	In this section, we focus on establishing Strichartz estimates for the Schrödinger propagator $e^{it\Phi(D)}$ associated with the operator $\Phi(D)$ in the framework of partial regularity. Before presenting the main result, we first recall a few preparatory lemmas. The following lemma concerns the fixed-ime decay estimate of the Schrödinger propagator $e^{it\Phi(D)}$ on mixed-norm Lebesgue spaces.
	\begin{lem}\label{eq6}
		Let $d \geq 2 $ and $1 \leq k \leq d$.  Let $\Phi \in C^{\infty}\left(\mathbb{R}^{d} \backslash\{0\}\right)$ be a real-valued function satisfying Assumptions (A).  Assume that $2 \leq \widetilde{r} \leq r \leq \infty$. Then, for $(x, y) \in \mathbb{R}^{d-k} \times \mathbb{R}^k,$ we have the following estimate:
		\begin{align}\label{eq5}
			\left\|e^{i t\Phi(D)} P_0 f\right\|_{L_x^r L_y^{\tilde{r}}} \lesssim(1+|t|)^{-\beta(r, \widetilde{r})}\|f\|_{L_x^{r^{\prime}} L_y^{\tilde{r}^{\prime}}}, 
		\end{align}
		where $P_0$ is the Littlewood-Paley operator given in \eqref{lpope} and $$\beta(r, \widetilde{r}):=(M-k)\left(\frac{1}{2}-\frac{1}{r}\right)+k\left(\frac{1}{2}-\frac{1}{\widetilde{r}}\right).$$
	\end{lem} 
	\begin{proof} In order to get the estimate (\ref{eq5}), we will apply  the Riesz-Thorin interpolation theorem. Thus, it is enough to obtain estimate (\ref{eq5})  for the following three cases:
 \begin{enumerate}[\textbf{Case} (a)]
 \item $r=\widetilde{r}=2$,
 \item  $r=\widetilde{r}=\infty$,
 \item  $r=\infty$ and $\tilde{r}=2$.
 \end{enumerate}
 
 It is clear that \textbf{Case} (a) follows directly from the Plancherel theorem of the Fourier transform. 
        
        For \textbf{Case} (b), that is, when  $r=\widetilde{r}=\infty$,  considering Fourier transform and its inversion formula, we can write
		$$
		\begin{aligned}
			e^{i t\Phi(D)} P_0 f(x, y) & =\int_{\mathbb{R}^{d-k}} \int_{\mathbb{R}^k} K\left(x-x^{\prime}, y-y^{\prime}, t\right) f\left(x^{\prime}, y^{\prime}\right) d y^{\prime} d x^{\prime} \\
			& =(K * f)(x,y),
		\end{aligned}
		$$
		where $P_0$ is given in \eqref{lpope} and the kernel $K$ is given by 
		\begin{align}\label{eq9}
			K(x, y, t)= \frac{2}{(2\pi)^d}\int_{\mathbb{R}^{d-k}} \int_{\mathbb{R}^k} e^{i(x, y) \cdot(\xi, \eta)} e^{i t \Phi(\xi, \eta)} \psi(\xi, \eta) d \eta d \xi.  
		\end{align}
		Using  Young's inequality, we get
		\begin{align}\label{Kernel estimate 1}
			\left\|e^{i t\Phi(D)} P_0 f\right\|_{L_{x, y}^{\infty}}=\left\|K * f\right\|_{L_{x, y}^{\infty}} \leq\left\|K\right\|_{L_{x, y}^{\infty}}\|f\|_{L_{x, y}^1} .
		\end{align}
		Now, we will estimate $L_{x, y}^{\infty}$  norm of the kernel $K$.    From our assumption on $\Phi$,  rank $H \Phi(\xi, \eta) \geq M$ on  the support $\left\{(\xi, \eta) \in \mathbb{R}^{d-k} \times \mathbb{R}^k: 1 / 2<|(\xi, \eta)|<2\right\}$ of $\psi$ in (\ref{eq9}),  by  Lemma \ref{eq8},    (\ref{eq9})  can be estimate in the following way:
		$$
		\begin{aligned}
			\left|K(x, y, t)\right| & =\left|\int_{\mathbb{R}^{d-k}} \int_{\mathbb{R}^k} e^{i(x, y, t) \cdot\left(\xi, \eta,\Phi(\xi, \eta)\right)} \psi(\xi, \eta) d \eta d \xi\right| \\
			& \lesssim(1+|(x, y, t)|)^{-\frac{M}{2}} \\
			& \lesssim(1+|t|)^{-\frac{M}{2}}.
		\end{aligned}
		$$
		Then from the inequality \eqref{Kernel estimate 1}, we get 
		$$		\left\|e^{i t\Phi(D)} P_0 f\right\|_{L_{x, y}^{\infty}}=\left\|K * f\right\|_{L_{x, y}^{\infty}}   \lesssim(1+|t|)^{-\beta(\infty, \infty)} \|f\|_{L_{x, y}^1} $$
		with $ \beta(\infty, \infty)=M / 2 \geq 0$ and   this concludes {\bf Case} (b). 
		
		For \textbf{Case} (c), that is, when $r=\infty$ and $\tilde{r}=2$.  Using Minkowski's inequality and Plancherel's theorem with respect to $y$ variable, we get 
		\begin{align}\label{Kernel 2}\nonumber
			\left\|e^{i t\Phi(D)} P_0 f\right\|_{L_x^{\infty} L_y^2} & =\left\|\| K * f \|_{L_y^2}\right\|_{L_x^{\infty}} \\\nonumber
			& \leq\left\|\int_{\mathbb{R}^{d-k}} \| K\left(x-x^{\prime}, \cdot\right) *_y f\left(x^{\prime}, \cdot\right) \|_{L_y^2} d x^{\prime}\right\|_{L_x^{\infty}} \\
			& =\left\|\int_{\mathbb{R}^{d-k}} \| \widetilde{K}\left(x-x^{\prime}, \cdot\right) \tilde{f}\left(x^{\prime}, \cdot\right) \|_{L_\eta^2} d x^{\prime}\right\|_{L_x^{\infty}},
		\end{align}
		where  $\tilde{f}=\mathcal{F}_y(f(x, \cdot))$ denotes the  Fourier transform of $f$ with respect to  the   $y \in \mathbb{R}^k$ variable. Let us denote    $\tilde{K}=\mathcal{F}_y(K(x, \cdot))$. Then    $\tilde{K}$  can be written explicitely as 
		$$
		\widetilde{K}(x, \eta, t)=\frac{1}{(2 \pi)^{d-k}} \int_{\mathbb{R}^{d-k}} e^{i (x, t) \cdot (\xi, \Phi(\xi, \eta))} \psi(\xi, \eta) d \xi . 
		$$
		Again applying Lemma \ref{eq8} with phase functions $\Phi(\xi, \eta)$ for fixed $|\eta| \leq 2$ with our assumption  that  rank $H{\Phi}(\xi, \eta) \geq M-k,$ we  get
		\begin{align}\label{eq7}
			\sup _{|\eta| \leq 2}\left|\widetilde{K}(x, \eta, t)\right| \lesssim(1+|(x, t)|)^{-\frac{M-k}{2}} .
		\end{align}
	Coming back to the inequality \eqref{Kernel 2},	using  the estimate  \eqref{eq7} and   Young's inequality,  we obtain 
		$$
		\begin{aligned}
			\left\|e^{i t\Phi(D)} P_0 f\right\|_{L_x^{\infty} L_y^2} &\leq	\left\|\int_{\mathbb{R}^{d-k}} \| \widetilde{K}\left(x-x^{\prime}, \cdot\right) \tilde{f}\left(x^{\prime}, \cdot\right) \|_{L_\eta^2} d x^{\prime}\right\|_{L_x^{\infty}}\\
			&	\lesssim 	\left\|\int_{\mathbb{R}^{d-k}}  (1+|(x-x', t)|)^{-\frac{M-k}{2}}\|   \tilde{f}\left(x^{\prime}, \cdot\right) \|_{L_\eta^2} d x^{\prime}\right\|_{L_x^{\infty}}\\ & \lesssim\left\|(1+|(\cdot, t)|)^{-\frac{M-k}{2}} *_x \| \tilde{f} \|_{L_\eta^2}\right\|_{L_x^{\infty}} \\
			& \lesssim(1+|t|)^{-\frac{M-k}{2}}\left \| \| \tilde{f} \|_{L_\eta^2}\right\|_{L_x^1} \\
			& \leq(1+|t|)^{-\frac{M-k}{2}}\|f\|_{L_x^1 L_y^2}\\
			& =(1+|t|)^{-\beta(\infty, 2)}\|f\|_{L_x^1 L_y^2} 
		\end{aligned}
		$$
		with $ \beta(\infty, 2)=\frac{M-k}{2} \geq 0$. This completes the proof of the lemma.  
	\end{proof}
	Next we assuming the above decay estimates, we have the following  frequency localized estimates
	\begin{lem} \label{Loc} Let $\Phi \in C^{\infty}\left(\mathbb{R}^{d} \backslash\{0\}\right)$ be a real-valued function satisfying Assumptions (A). Assume that $d \geq 1,1 \leq k \leq d$, and
		\begin{align*}
			\frac{2}{q} \leq(M-k)\left(\frac{1}{2}-\frac{1}{r}\right)+k\left(\frac{1}{2}-\frac{1}{\widetilde{r}}\right), \quad 2 \leq \widetilde{r} \leq r<\infty, \quad 2\leq q \leq \infty. 
		\end{align*}
		Then, we have  
		\begin{align}\label{eq4}
			\left\|e^{i t\Phi(D)} P_0 f\right\|_{L_t^q L_x^r L_y^{\tilde{\tau}}} \lesssim\|f\|_{L_{x, y}^2}.
		\end{align}
		
	\end{lem}
	\begin{proof}
		By   a standard $T T^*$ argument,  the required estimate (\ref{eq4}) is equivalent to
		$$
		\left\|\int_{\mathbb{R}} e^{i(t-\tau)\Phi(D)} P_0 g(\tau,\cdot) d \tau\right\|_{L_t^q L_x^r L_y^{\tilde{r}}} \lesssim\|g\|_{L_t^{q^{\prime}} L_x^{r^{\prime}} L_y^{\tilde{r}^{\prime}}}.
		$$
		Thus our aim is to prove the above estimate. By  the estimate \eqref{eq5} in Lemma \ref{eq6}, we obtain
		$$
		\begin{aligned}
			\left\|\int_{\mathbb{R}} e^{i(t-\tau)\Phi(D)} P_0 g(\tau,\cdot) d \tau\right\|_{L_t^q L_x^r L_y^{\tilde{r}}} & \leq\left\|\int_{\mathbb{R}} \| e^{i(t-\tau)\Phi(D)} P_0 g(\tau,\cdot) \|_{L_x^r L_y^{\tilde{r}}} d \tau\right\|_{L_t^q} \\
			& \lesssim\left\|\int_{\mathbb{R}}  (1+|t-\tau|)^{-\beta(r, \widetilde{r})}\|g(\tau,\cdot)\|_{L_x^{r^{\prime}} L_y^{\tilde{r}^{\prime}}} d \tau\right\|_{L_t^q} \\
			& \lesssim\left\|(1+|\cdot|)^{-\beta(r, \widetilde{r})} *_t \| g \|_{L_x^{r^{\prime}} L_y^{\tilde{r}^{\prime}}}\right\|_{L_t^q}.
		\end{aligned}
		$$
		From the notation of $\beta(r, \tilde{r})$,  the inequality $	\frac{2}{q} \leq(M-k)\left(\frac{1}{2}-\frac{1}{r}\right)+k\left(\frac{1}{2}-\frac{1}{\widetilde{r}}\right)$  reduces to   $2 / q \leq \beta(r, \tilde{r})$. We now consider three cases.
		\begin{itemize}
			\item When $2 / q<\beta(r, \tilde{r})$ and $2 \leq q \leq \infty$:    Then Young's inequality yields
			$$
			\begin{aligned}
				\left\|(1+|\cdot|)^{-\beta(r, \widetilde{r})} *_t \| g \|_{L_x^{r^{\prime}} L_y^{\widetilde{r}}}\right\|_{L_t^q} & \lesssim\left\|(1+|\cdot|)^{-\beta(r, \widetilde{r})}\right\|_{L_t^q}\|g\|_{L_t^{q^{\prime}} L_x^{r^{\prime}} L_y^{\widetilde{r}}} \\
				& \lesssim\|g\|_{L_t^{q^{\prime}} L_x^{r^{\prime}} L_y^{\widetilde{r}^{\prime}}}.
			\end{aligned}
			$$
			\item 	When $2 / q=\beta(r, \tilde{r})$ with $2<q<\infty$: Applying one-dimensional Hardy-Littlewood Sobolev inequality, we get 
			$$
			\begin{aligned}
				\left\|(1+|\cdot|)^{- \beta(r, \tilde{r})} *_t \| g \|_{L_x^{r^{\prime}} L_y^{\tilde{r}^{\prime}}}\right\|_{L_t^q} & \lesssim\left\||\cdot|^{- \beta(r, \tilde{r})} *_t \| g \|_{L_x^{r^{\prime}} L_y^{\tilde{r}^{\prime}}}\right\|_{L_t^q} \\
				& \lesssim\|g\|_{L_t^{q^{\prime}} L_x^{\tau_x^{\prime}} L_y^{\tilde{r}^{\prime}}}.
			\end{aligned}
			$$
			\item 	When  $q=\infty$ and $ \beta(r, \widetilde{r})=0$: This implies  that  $r=\widetilde{r}=2$. In this case, the required estimate (\ref{eq5})   holds trivially by Plancherel's theorem. This completes the proof of the lemma.
		\end{itemize} \end{proof}

	Now we are in a position to prove  Strichartz estimates in the partial regularity framework.
	\begin{theorem} \label{eq13} Let $\Phi \in C^{\infty}\left(\mathbb{R}^{d} \backslash\{0\}\right)$ be a real-valued function satisfying Assumptions (A).
		Let $d \geq 1,  1 \leq k \leq d, 2 \leq \widetilde{r} \leq r<\infty $ and $ 2<q \leq \infty$. If $ q, r, \tilde{r}$ satisfy 
		\begin{align}\label{eq0}
			\frac{2}{q} \leq(M-k)\left(\frac{1}{2}-\frac{1}{r}\right)+k\left(\frac{1}{2}-\frac{1}{\widetilde{r}}\right). 
		\end{align}
		Then we have
		\begin{align}\label{eq1}
			\left\|e^{i t\Phi(D)} f\right\|_{L_t^q\left(\mathbb{R} ; L_x^r\left(\mathbb{R}^{d-k} ; L_y^{\tilde{r}}\left(\mathbb{R}^k\right)\right)\right)} \lesssim\|f\|_{\dot{H}_{x, y}^s}
		\end{align}
		under the scaling condition
		\begin{align}\label{eq2}
			\frac{{{m}}}{q}=-s +(d-k)\left(\frac{1}{2}-\frac{1}{r}\right)+k\left(\frac{1}{2}-\frac{1}{\widetilde{r}}\right),
		\end{align} where $m \geq 1$ is the homogeneous degree of $\Phi.$
		
	\end{theorem}   
	\begin{proof} We begin the proof by substituting $e^{i t\Phi(D)} f$  into the estimate  \eqref{norm} as follows  $$
		\|e^{i t\Phi(D)} f\|_{L_x^r L_y^{\tilde{r}}} \lesssim \left\| \left(\sum_{j \in \mathbb{Z}} \left|P_j e^{i t\Phi(D)} f\right|^2\right)^{1 / 2}\right\|_{L_x^r L_y^{\tilde{r}}}. 
		$$ Next, using the Minkowski inequality, we can  conclude that 
		\begin{align}\label{eq12}\nonumber
			\left\|e^{i t\Phi(D)} f\right\|_{L_t^q L_x^r L_y^{\tilde{r}}}^2 &\lesssim\left\| \left( \sum_{j \in \mathbb{Z}} \left|P_j e^{i t\Phi(D)} f\right|^2\right)^{1 / 2}\right\|_{L_t^q L_x^r L_y^{\tilde{r}}}^2 \\
			& \lesssim \sum_{j \in \mathbb{Z}}\left\|e^{i t\Phi(D)} P_j f\right\|_{L_t^q L_x^r L_y^{\tilde{r}}}^2,
		\end{align}
        where in the last step, we used the fact that  $P_j$  commutes with the semigroup $e^{i t\Phi(D)}$.		Now, applying the Fourier transform and inversion, we can write 
		$$
		\begin{aligned}
			e^{i t\Phi(D)} P_j f(x, y) & = \int_{\mathbb{R}^{d}} e^{i\left(x,  y\right) \cdot(\xi, \eta)+it  \Phi(\xi, \eta)} \psi_j(\xi, \eta) \hat{f}\left( \xi,  \eta\right)   d \eta d \xi \\
			&=  \int_{\mathbb{R}^{d}} e^{i\left(x,  y\right) \cdot(\xi, \eta)+it  \Phi(\xi, \eta)} \psi(2^{-j}\xi, 2^{-j}\eta) \hat{f}\left( \xi,  \eta\right)   d \eta d \xi .
		\end{aligned}
		$$
   By the  change of variables $2^{-j} \xi \rightarrow \xi$ and $2^{-j} \eta \rightarrow \eta$ and using the   homogeneity  property of $\Phi$ as given in the Assumption (A), we get 
		$$
		\begin{aligned}
			e^{i t\Phi(D)} P_j f(x, y) & =  \int_{\mathbb{R}^{d}} e^{i\left(2^jx,  2^jy\right) \cdot(\xi, \eta)+it  \Phi(2^j\xi, 2^j\eta)}  \psi(\xi, \eta) \hat{f}\left( 2^j\xi,  2^j\eta\right)   2^{jN}d \eta d \xi \\
			& =  \int_{\mathbb{R}^{d}} e^{i\left(2^jx,  2^jy\right) \cdot(\xi, \eta)+it  \Phi(2^{j}(\xi, \eta))}  \psi(\xi, \eta) \hat{f}\left( 2^j\xi,  2^j\eta\right)   2^{jN}d \eta d \xi \\
			& =  \int_{\mathbb{R}^{d}} e^{i\left(2^jx,  2^jy\right) \cdot(\xi, \eta)+i2^{mj} t \Phi(\xi, \eta)}  \psi(\xi, \eta) \hat{f}\left( 2^j\xi,  2^j\eta\right)   2^{jN}d \eta d \xi \\
			&
			=e^{i 2^{jm} t\Phi(D)} P_0 f_j\left(2^j x, 2^j y\right)
		\end{aligned}
		$$
  where $f_j(x, y):=f\left(2^{-j} x, 2^{-j} y\right)$.  Now using the localized estimate (\ref{eq4}) given in Lemma \ref{Loc} and the above scaling estimates,  we can write 
		\begin{align}\label{eq141}\nonumber
			\left\|e^{i t\Phi(D)} P_j f\right\|_{L_t^q L_x^r L_y^{\tilde{r}}}^2 
			&   \lesssim 2^{j\left(-\frac{{m}}{q}-\frac{d-k}{r}-\frac{k}{\tilde{r}}\right)}\left\|e^{i t\Phi(D)} P_0 f_j\right\|_{L_t^q L_x^r L_y^{\tilde{r}}} \\\nonumber
			& \lesssim 2^{j\left(-\frac{{{m}}}{q}-\frac{d-k}{r}-\frac{k}{\tilde{r}}\right)}\left\|f_j\right\|_{L_{x, y}^2} \\\nonumber
			& \lesssim  2^{-\frac{jd}{2}}2^{js}2^{j\left(-s-\frac{{{m}}}{q}-\frac{d-k}{r}-\frac{k}{\tilde{r}}+\frac{d}{2}\right)}\left\|f_j\right\|_{L_{x, y}^2} \\
			& \leq  2^{j s}\|f\|_{L_{x, y}^2},
		\end{align}
		where we used the scaling condition \eqref{eq2} in the last inequality.  
        
        Now, from  the quasi-orthogonal relations $P_jP_k=0$ if $|j-k|>1$ and the identity relation $f= \displaystyle \sum_{j \in \mathbb{Z}} P_j f$,   if we write $P_j f=P_j \widetilde{P}_j f$   with $\widetilde{P}_j=\sum_{k:|j-k| \leq 1} P_k$, then   from    (\ref{eq12}) and \eqref{eq141}, we obtain  
		$$
		\begin{aligned}
			\left\|e^{i t\Phi(D)} f\right\|_{L_t^q L_x^r L_y^{\tilde{r}}}^2  
			& \lesssim \sum_{j \in \mathbb{Z}}\left\|e^{i t\Phi(D)} P_j f\right\|_{L_t^q L_x^r L_y^{\tilde{r}}}^2 \\
			&= \sum_{j \in \mathbb{Z}}\left\|e^{i t\Phi(D)} P_j \widetilde{P}_j f\right\|_{L_t^q L_x^r L_y^{\tilde{\tau}}}^2\\
			&   \lesssim \sum_{j \in \mathbb{Z}} 2^{2 j s}\left\|\widetilde{P}_j f\right\|_{L_{x, y}^2}^2 \lesssim\|f\|_{\dot{H}_{x, y}^s}^2,
		\end{aligned}
		$$
		completing the proof of the theorem. 
		\end{proof}

    
	 By choosing specific forms of the function $\Phi$, we arrive at particular Strichartz estimates for partially regular initial data. To start, we look at the Strichartz estimate for the free Schrödinger propagator, through which we recover \cite[Theorem 1.3]{KLS23}.
     
    \begin{cor} Let $d \geq 1,$ $1 \leq k \leq d$ and \begin{align*}
				\frac{2}{q} \leq(d-k)\left(\frac{1}{2}-\frac{1}{r}\right)+k\left(\frac{1}{2}-\frac{1}{\widetilde{r}}\right), \quad 2 \leq \widetilde{r} \leq r<\infty, \quad 2<q \leq \infty. 
			\end{align*} Then, we have the following Strichartz inequality 
      
			\begin{align*}
				\left\|e^{-i t \Delta} f\right\|_{L_t^q\left(\mathbb{R} ; L_x^r\left(\mathbb{R}^{d-k} ; L_y^{\tilde{r}}\left(\mathbb{R}^k\right)\right)\right)} \lesssim \|f\|_{\dot{H}_{x, y}^s},
			\end{align*}
			under the scaling condition
			\begin{align*}
				\frac{{{2}}}{q}=-s+(d-k)\left(\frac{1}{2}-\frac{1}{r}\right)+k\left(\frac{1}{2}-\frac{1}{\widetilde{r}}\right).
			\end{align*}
    \end{cor}
    \begin{proof} By choosing $\Phi(\xi)=|\xi|^2$, we can verify that it satisfies Assumption (A) for $m=2, M=d.$ Then, the result immediately follows from  Theorem \ref{eq13}.  
    \end{proof}

We show that our result yields the Strichartz estimate for the wave propagator, which was already established in \cite[Theorem 1.3]{KLS23}.

\begin{cor} Let $d \geq 2,$ $1 \leq k \leq d-1$ and \begin{align*}
				\frac{2}{q} \leq(d-1-k)\left(\frac{1}{2}-\frac{1}{r}\right)+k\left(\frac{1}{2}-\frac{1}{\widetilde{r}}\right), \quad 2 \leq \widetilde{r} \leq r<\infty, \quad 2<q \leq \infty.
			\end{align*}
Then, the following estimate 
			\begin{align*}
				\left\|e^{i t \sqrt{-\Delta}} f\right\|_{L_t^q\left(\mathbb{R} ; L_x^r\left(\mathbb{R}^{d-k} ; L_y^{\tilde{r}}\left(\mathbb{R}^k\right)\right)\right)} \lesssim \|f\|_{\dot{H}_{x, y}^s}
			\end{align*}
			holds  
			under  the scaling condition
			\begin{align*}
				\frac{{{1}}}{q}=(d-k)\left(\frac{1}{2}-\frac{1}{r}\right)+k\left(\frac{1}{2}-\frac{1}{\widetilde{r}}\right)-s.
			\end{align*}
\begin{proof} We consider  $\Phi(\xi)=|\xi|$, then $\Phi$ satisfies Assumption (A) with  $m=1, M=d-1$,  and  for any $1 \leq k \leq d-1$. Hence, Theorem \ref{eq13} gives the desired result.   \end{proof}
             \end{cor}

Our next result concerns the Strichartz estimates for the biharmonic operator with partially regular initial data. This particular case appears to be new in the literature.
            \begin{cor} \label{biha} Let $d \geq 1,$ $1 \leq k \leq d$ and \begin{align*}
				\frac{2}{q} \leq(d-k)\left(\frac{1}{2}-\frac{1}{r}\right)+k\left(\frac{1}{2}-\frac{1}{\widetilde{r}}\right), \quad 2 \leq \widetilde{r} \leq r<\infty, \quad 2<q \leq \infty. 
			\end{align*}
            Then we have 
			\begin{align*}
				\left\|e^{-i t \Delta^2} f\right\|_{L_t^q\left(\mathbb{R} ; L_x^r\left(\mathbb{R}^{d-k} ; L_y^{\tilde{r}}\left(\mathbb{R}^k\right)\right)\right)} \lesssim \|f\|_{\dot{H}_{x, y}^s},
			\end{align*}
			under  the scaling condition
			\begin{align*}
				\frac{{{4}}}{q}=(d-k)\left(\frac{1}{2}-\frac{1}{r}\right)+k\left(\frac{1}{2}-\frac{1}{\widetilde{r}}\right)-s.
			\end{align*}
	           \end{cor}
               \begin{proof} To derive this result, we choose  $\Phi(\cdot)=|\cdot|^4.$	 We write   $\Phi(\xi, \eta)=(|\xi|^2+|\eta|^2)^{2}$, where $\xi=(\xi_1, \cdots, \xi_{d-k})\in \mathbb{R}^{d-k}$ and $\eta=(\eta_1, \cdots, \eta_{k})\in \mathbb{R}^{k}.$  In this case  the Hessian matrix $H \Phi$ is given by 
               $$
					H \Phi  (\xi, \eta ) =  
					\begin{pmatrix}
						8(	|\xi|^2 + |\eta|^2) +16\xi_1^2  &   \cdots  &   16\xi_1 \eta_{k} \\
						\vdots     &   \ddots    &   \vdots \\
						16\xi_{d-k} \eta_{1}    &   \cdots &   	8(	|\xi|^2 + |\eta|^2)  +16\eta_{k}^2
					\end{pmatrix} .
					$$
                    Observe that the rank    of $H \Phi  (\xi, \eta ) =d$ for all $(\xi, \eta) \neq 0.$
                    Again, if we consider $\Phi_\eta  (\xi )=(|\xi|^2+|\eta|^2)^{2}$ for fixed     	 $|\eta| \leq 2$. Then Hessian matrix $H \Phi_\eta $ is given by   	$$H \Phi_\eta  (\xi ) =  
					\begin{pmatrix}
						4(	|\xi|^2 + |\eta|^2) +8\xi_1^2  &   \cdots  &   8\xi_1 \xi_{d-k} \\
						\vdots     &   \ddots    &   \vdots \\
						8\xi_{d-k} \xi_{1}    &   \cdots &   	4(	|\xi|^2 + |\eta|^2)  +8\xi_{d-k}^2
					\end{pmatrix} .
					$$
					in this case rank of $H \Phi_\eta  (\xi ) =d-k$ for fixed     	 $|\eta| \leq 2$.
               Thus, Assumptions (A) is verified by $\Phi$ by taking      $m=4, M=d$,  and   for any $1 \leq k \leq d.$  Then, this corollary follows as a direct application of Theorem \ref{eq13}.
               \end{proof}

               The next result is about the Strichartz estimates for the higher order Schr\"odinger operators with partial regular initial data. 

        \begin{cor} \label{hos} Let $d \geq 1,$ $\alpha>2,$ $1 \leq k \leq d$ and 
        \begin{align}\label{Condition1}
				\frac{2}{q} \leq(d-k)\left(\frac{1}{2}-\frac{1}{r}\right)+k\left(\frac{1}{2}-\frac{1}{\widetilde{r}}\right), \quad 2 \leq \widetilde{r} \leq r<\infty, \quad 2<q \leq \infty,.
			\end{align}
            Then we have
			\begin{align}\label{eq11111}
				\left\|e^{i t (-\Delta)^{\frac{\alpha}{2}}} f\right\|_{L_t^q\left(\mathbb{R} ; L_x^r\left(\mathbb{R}^{d-k} ; L_y^{\tilde{r}}\left(\mathbb{R}^k\right)\right)\right)} \lesssim \|f\|_{\dot{H}_{x, y}^s}
			\end{align}
			under  the scaling condition
			\begin{align}\label{Condition2}
				\frac{{{\alpha}}}{q}=(d-k)\left(\frac{1}{2}-\frac{1}{r}\right)+k\left(\frac{1}{2}-\frac{1}{\widetilde{r}}\right)-s.
			\end{align}
			Moreover,  the  following  inhomogeneous estimate holds
			\begin{align}\label{inhom}
				\left\|\int_0^t e^{i(t-\tau)  (-\Delta)^{\frac{\alpha}{2}}} F_p(u) d \tau\right\|_{L_t^q\left(\mathcal{I} ; L_x^r\left(\mathbb{R}^{d-k} ; L_y^{\tilde{r}}\left(\mathbb{R}^k\right)\right)\right)}  \leq C\left\|F_p\right\|_{L_t^{q_1'}\left(\mathcal{I} ; L_x^{r_1'}\left(\mathbb{R}^{d-k} ; L_y^{\tilde{r_1}'}\left(\mathbb{R}^k\right)\right)\right)}
			\end{align}
			for any pairs $(q, r, \tilde{r}),(q_1, r_1, \tilde{r_1}) $ that satisfies (\ref{Condition1})  and (\ref{Condition2}), where $a'$ denotes the conjugate of $a $, i.e., $\frac{1}{a}+\frac{1}{a'}=1$.
			
      \end{cor}
\begin{proof} The proof of this result is similar to the proof of Corollary \ref{biha} by choosing   $\Phi(\xi)=|\xi|^\alpha, \alpha>2$ with $m=\alpha, M=d$,  and   for any $1 \leq k \leq d.$ The inhomogeneous estimate follows from standard  $TT^*$ argument and the dual norm principle.
\end{proof}
    Finally,   choose $s=0$ and $k=2$ in Corollary \ref{hos} we get the following result, which coincides with the classical Strichartz estimates given in \cite{CHL11} for $r=\tilde{r}$.

	\begin{cor} Let $d \geq 2, \alpha >2$ and  
    \begin{align}\label{a}
			\frac{2}{q} \leq(d-2)\left(\frac{1}{2}-\frac{1}{r}\right)+2\left(\frac{1}{2}-\frac{1}{\widetilde{r}}\right), \quad 2 \leq \widetilde{r} \leq r<\infty, \quad 2<q \leq \infty.
		\end{align}
        Then we have 
		\begin{align*}
			\left\|e^{i t (-\Delta)^{\frac{\alpha}{2}}}  f\right\|_{L_t^q\left(\mathbb{R} ; L_x^r\left(\mathbb{R}^{d-2} ; L_y^{\tilde{r}}\left(\mathbb{R}^2\right)\right)\right)} \lesssim\|f\|_{L^2(\mathbb{R}^{d})}
		\end{align*}
		under  the scaling condition
		\begin{align*}
			\frac{{{\alpha}}}{q}=(d-2)\left(\frac{1}{2}-\frac{1}{r}\right)+2\left(\frac{1}{2}-\frac{1}{\widetilde{r}}\right).
		\end{align*}
		\end{cor}
	
	\section{Application: Local well-posedness of higher order nonlinear Schr\"odinger equations}\label{sec4}
	In this section, we prove the local well-posedness theorem for the higher order nonlinear Schr\"odinger equation (\ref{higher order}) with partially regular initial data $f \in L_x^2\left(\mathbb{R}^{d-2} ; H_y^s\left(\mathbb{R}^2\right)\right)$. We begin by describing our approach to prove the local wellposedness  of the nonlinear problem \eqref{eq10} for $\Phi(\xi)=|\xi|^m$, that is, equation \eqref{higher order}. \\
	 
	\textbf{Philosophy of our approach:} First, we note that the solution of the linear equation (\ref{higher order}) with $F=0$ is given by $e^{i t(-\Delta)^{\frac{m}{2}} } f$. By Duhamel's principle, the solution map of the nonlinear equation  (\ref{higher order})   can be written as 
	$$
	u(t, x)=e^{i t (-\Delta)^{\frac{m}{2}}} f(x)-i \int_0^t e^{i(t-\tau)(-\Delta)^{\frac{m}{2}}} F_p(u)(t, x)\; d \tau. 
	$$ 
	For $0<s \leq 1$ and suitable values of $T, A>0$, we introduce the evolution space 
	\begin{equation} \label{evospace}
	    X(T)=\left\{u \in C\left([0, T] ; L_x^2 H_y^s\right) \cap L_t^q\left([0, T] ; L_x^r W_y^{s, \widetilde{r}}\right):  \|u\|_{L_t^q\left([0, T] ; L_x^r W_y^{s, \tilde{r}}\right)} \leq A\right\}
	\end{equation} with the distance
	
	$$
	d(u, v)=\|u-v\|_{L_t^q\left([0,T] ; L_x^r L_y^{\tilde{r}}\right)}. 
	$$ 	
	We define the  operator  $N: X(T)\to X(T)$ defined  by $$Nu(t, x):=e^{i t (-\Delta)^{\frac{m}{2}}} f(x)-i \int_0^t e^{i(t-\tau) (-\Delta)^{\frac{m}{2}}} F_p(u)(\tau, x)\; d \tau.$$
	We aim to prove the local existence and uniqueness of small data Sobolev solutions of low regularity to the  equation (\ref{higher order}) with
	the help of the Banach's  fixed point theorem argument.  We find a  unique fixed point (say) $u^*$ of the operator $N$, that means, $u^*=N u^* \in X(T)$ for all positive $T$. More precisely, to find such a unique fixed point,   we will establish two crucial inequalities of the form
	\begin{align}\label{Banach1}
		\|N u\|_{X(T)} & \lesssim  \left\|f\right\|_{L_x^2\left(\mathbb{R}^{d-2} ; H_y^s\left(\mathbb{R}^2\right)\right)}+\|u\|_{X(T)}^p  \end{align}
	and \begin{align}\label{Banach2}
		d(u, v)=\|Nu-Nv\|_{L_t^q\left(I ; L_x^r L_y^{\tilde{r}}\right)}& \lesssim\|u-v\|_{X(T)}\left[  \|u\|_{X(T)}^{p-1}+\|v\|_{X(T)}^{p-1}\right],
	\end{align}
	for any $u, v \in X(T)$.

	The following result is  related to the local in-time wellposedness  for the higher order nonlinear Schrödinger equation (\ref{higher order}) with partially regular initial data $f \in L_x^2\left(\mathbb{R}^{d-2} ; H_y^s\left(\mathbb{R}^2\right)\right)$. 
	\begin{theorem}
		For   $0<s \leq 1, m>2$ and $d \geq 3$,  let   $1<p<1+\frac{2m}{ d-2s}$ and $f \in$ $L_x^2\left(\mathbb{R}^{d-2} ; H_y^s\left(\mathbb{R}^2\right)\right)$.  Further, assume that  
		$(q, r, \tilde{r})$ with $2 \leq \widetilde{r} \leq r<\infty$ and $ 2<q \leq \infty $ satisfies $$			\frac{{{m}}}{q}=(d-2)\left(\frac{1}{2}-\frac{1}{r}\right)+2\left(\frac{1}{2}-\frac{1}{\widetilde{r}}\right).$$ Then there exists  $T>0$ 
		$$
		u \in C_t\left([0, T] ; L_x^2\left(\mathbb{R}^{d-2}; H_y^s(\mathbb{R}^{2})\right)\right) \cap   L_t^q\left([0, T] ; L_x^r \left(\mathbb{R}^{d-2}; W_y^{s, \widetilde{r}}(\mathbb{R}^{2})\right)\right)
		$$
		that satisfies  the higher order nonlinear schr\"odinger equation  (\ref{higher order}).
	\end{theorem}
	\begin{proof}

		In order to prove the existence of a local solution to the nonlinear equation \eqref{higher order}, it is enough to show that the map $$N: X(T)\to Nu(t, x):=e^{i t (-\Delta)^{\frac{m}{2}}} f(x)-i \int_0^t e^{i(t-\tau) (-\Delta)^{\frac{m}{2}}} F_p(u)(\tau, x)\; d \tau$$
		is defines a contraction map on $X(T)$ defined in \eqref{evospace}. To prove $N$ is a contraction map on $X(T),$ by denoting $\mathcal{I}= [0,T]$,  our first step is to show  that
		\begin{align*}
			\left\|\left\langle\nabla_y\right\rangle^s F_p(u)\right\|_{L_t^{q_1^{\prime}}\left(\mathcal{I} ; L_x^{r_1^{\prime}} L_y^{\tilde{r}_1^{\prime}}\right)} \leq C T^{\beta(N, p, s)}\|u\|_{L_t^{q_1}\left(\mathcal{I} ; L_x^{r_1} W_y^{s, \tilde{r}_1}\right)}^p
		\end{align*}
		for $0<s \leq 1$ and $1<p<1+\frac{2m}{d-2s}$ and for  some  $\beta(N, p, s)>0$.

		For $d \geq 3$, $0<s \leq 1$ and for a fixed exponent $p$ with  $1<p<1+\frac{2m}{d-2s}$, we   choose   $\varepsilon>0$ and  a pair $\left(q_1, r_1, \widetilde{r}_1\right)$ such that it satisfies the following conditions  
		$$
		\frac{(1-s)(p-1)}{2}<\varepsilon<\min \left\{\frac{d-2}{m^2}\left(1+\frac{2m}{d-2}-p\right), \frac{p-1}{2}\right\}
		$$
		and 
		\begin{equation} \label{mxm}
		    \frac{1}{r_1}=\frac{1}{p+1}, \quad \frac{1}{\widetilde{r}_1}=\frac{1}{2}-\frac{\epsilon}{p+1}, \quad \frac{1}{q_1}=\frac{d-2}{2m}-\frac{d-2}{m(p+1)}+\frac{m\varepsilon}{2(p+1)}.
		\end{equation}
		Note that above   $\left(q_1, r_1, \widetilde{r}_1\right)$ also satisfies the admissible condition 	$\frac{{{m}}}{q_1}+\frac{d-2}{r_1}+\frac{2}{\widetilde{r}_1}=\frac{d}{2}.$
        
        Using \eqref{nonlin}  and the H\"older inequality with \eqref{mxm}, we get the estimate 
		\begin{align}\label{Holder1}\nonumber
			\left\|F_p(u)\right\|_{L_t^{q_1^{\prime}}\left(\mathcal{I} ; L_x^{r_1^{\prime}} L_y^{\tilde{r}_1^{\prime}}\right)} & \lesssim\left\|\chi_{\mathcal{I}}(t)|u|^{p-1} | u|\right\|_{L_t^{q_1^{\prime}}\left(\mathcal{I} ; L_x^{r_1^{\prime}} L_y^{L_1^{\prime}}\right)} \\
			&\leq C T^{1-\frac{p+1}{q_1}}\|u\|_{L_t^{q_1}\left(\mathcal{I} ; L_x^{r_1} L_y^{(p-1)(p+1) / 2 \varepsilon}\right)}^{p-1}\|u\|_{L_t^{q_1}\left(\mathcal{I} ; L_x^{r_1} L_y^{\tilde{r_0}}\right)}.
		\end{align}
	In a similar manner, for $0<s \leq 1$, using the following fractional chain rule for Laplacian on $\mathbb{R}^d$ $$\||\nabla|^s F(u)\|_{L^c}\leq \||F'(u)\|_{L^a}\||\nabla|^sF(u)\|_{L^c},$$
    where $\frac{1}{c}=\frac{1}{a}+\frac{1}{b},$ along with the H\"older inequality,  we obtain
		\begin{align}\label{Holder2}
			\left\|\left|\nabla_y\right|^s F_p(u)\right\|_{L_t^{q_1^{\prime}}\left(\mathcal{I} ; L_x^{r_1^{\prime}} L_y^{\tilde{r}_1}\right)} \leq C T^{1-\frac{p+1}{q_1}}\|u\|_{L_t^{q_1}\left(\mathcal{I} ; L_x^{r_1} L_y^{(p-1)(p+1) / 2 \varepsilon}\right)}^{p-1}\|u\|_{L_t^{q_1}\left(\mathcal{I} ; L_x^{r_1} W_y^{s, \tilde{r}_1}\right)}.
		\end{align}
		Using the fact that  $W^{s, \widetilde{r}_1} \subseteq L^{\widetilde{r}_1}$ for $s \geq 0$ and  by  the Sobolev embedding in two dimensions
		$$
		W^{s, \widetilde{r}_1} \hookrightarrow W^{1-\frac{2 \varepsilon}{p-1}, \widetilde{r}_1} \hookrightarrow L^{\frac{(p-1)(p+1)}{2 \epsilon_1}},
		$$
		the inequalities \eqref{Holder1} and \eqref{Holder2}  further   reduces to 
		\begin{align}\label{using embedding }
			\left\||\nabla_y|^s F_p(u)\right\|_{L_t^{q_1^{\prime}}\left(\mathcal{I} ; L_x^{r_1^{\prime}} L_y^{\vec{r}^{\prime}}\right)} \leq C T^{\beta_1(d, p, s)}\|u\|_{L_t^{q_1}\left(\mathcal{I} ; L_x^{r_1} W_y^{s, \bar{r}_1}\right)}^p
		\end{align}
		with $1 \geq  s \geq 0$ and $\beta_1(d, p, s)=1-\frac{p+1}{q_1} =\frac{d-2}{2m}\left(1+\frac{2m}{d-2}-p\right)-\frac{m\varepsilon}{2}>0.$
		Thus using the inequalities \eqref{Holder1} and \eqref{using embedding }, we can write 
		\begin{align}\label{For banach}\nonumber
			\left\|\left\langle\nabla_y\right\rangle^s F_p(u)\right\|_{L_t^{q_1^{\prime}}\left(\mathcal{I} ; L_x^{r_1^{\prime}} L_y^{\tilde{r}_1^{\prime}}\right)} &\lesssim\left\|F_p(u)\right\|_{L_t^{q_1^{\prime}}\left(\mathcal{I} ; L_x^{r_1^{\prime}} L_y^{\tilde{r}_1}\right)}+\left\|\left|\nabla_y\right|^s F_p(u)\right\|_{L_t^{q_1^{\prime}}\left(\mathcal{I} ; L_x^{r_1^{\prime}} L_y^{\tilde{r}_1^{\prime}}\right)} \\
			&\lesssim  T^{\beta_1(d, p, s)}  \|u\|_{L_t^{q_1}\left(\mathcal{I} ; L_x^{r_1} W_y^{s, \bar{r}_1}\right)}^p.
		\end{align}
		In a similar way, we can have
		\begin{align}\label{To prove 2}\nonumber
			& \left\|F_p(u)-F_p(v)\right\|_{L_t^{q_1^{\prime}}\left(\mathcal{I}; L_x^{r_1^{\prime}} L_y^{\tilde{r_1}'}\right)} \\
			& \leq C T^{\beta_1(d, p, s)}\left(\|u\|_{L_t^{q_1}\left(\mathcal{I}; L_x^{r_1} W_y^{s, \tilde{r}_1}\right)}^{p-1}+\|v\|_{L_t^{q_1}\left(\mathcal{I} ; L_x^{r_1} W_y^{s, \tilde{r}_1}\right)}^{p-1}\right)\|u-v\|_{L_t^{q_1}\left(\mathcal{I}; L_x^{r_1} L_y^{\tilde{r}_1}\right)},
		\end{align} 
        by using the following elementary properties for the nonlinearity $F_p(u):$
        \begin{equation}
            |F_p(u)-F_p(v)| \leq C (|u|^{p-1}+|v|^{p-1}) |u-v|.
        \end{equation}
		Now our aim is to establish \eqref{Banach1} and \eqref{Banach2}. 
		By Plancherel's theorem, a dual version of  the Strichartz inequality \eqref{eq11111} with $s=0$, and from the inequality  \eqref{For banach}, one can see that 
		$$
		\begin{aligned}
			\sup _{t \in \mathcal{I}}\|Nu \|_{L_x^2 H_y^s} & \leq 
			\sup _{t \in \mathcal{I}} 	\|	e^{i t (-\Delta)^{\frac{m}{2}}} f \|_{L_x^2 H_y^s} +	\sup _{t \in \mathcal{I}}\left \| \int_0^t e^{i(t-\tau) (-\Delta)^{\frac{m}{2}}} F_p(u)\; d \tau \right\|_{L_x^2 H_y^s}\\ 		
			& \leq C\|f\|_{L_x^2 H_y^s}+C	\sup _{t \in \mathcal{I}} \left\|\left\langle\nabla_y\right\rangle^s \int_0^t e^{-i \tau (-\Delta)^{\frac{m}{2}}} F_p(u) d \tau\right\|_{L_x^2 L_y^2} \\
			& \leq C\|f\|_{L_x^2 H_y^s}+C\left\|\left\langle\nabla_y\right\rangle^s F_p(u)\right\|_{L_t^{q_1^{\prime}}\left(\mathcal{I} ; L_x^{r_1^{\prime}} L_y^{\tilde{r}_1^{\prime}}\right)}\\
			&	\leq C\|f\|_{L_x^2 H_y^s}+C	T^{\beta_1(d, p, s)}  \|u\|_{L_t^{q_1}\left(\mathcal{I} ; L_x^{r_1} W_y^{s, \bar{r}_1}\right)}^p,
		\end{aligned}
		$$
		for some $\left(q_1, r_1, \widetilde{r}_1\right) $  that satisfies 		$\frac{{{m}}}{q_1}+\frac{d-2}{r_1}+\frac{2}{\widetilde{r}_1}=\frac{d}{2}.$  Thus for $u \in X$, it follows that
		\begin{equation}\label{final1}
		    \sup _{t \in \mathcal{I}}\|Nu \|_{L_x^2 H_y^s} 	\leq C\|f\|_{L_x^2 H_y^s}+C	T^{\beta_1(d, p, s)} A^p.
		\end{equation}	 
		
		Let us now proceed to prove our next aim to show inequality of the form \eqref{Banach2}.  	 Now using  
		\eqref{inhom} and \eqref{To prove 2}, we get 
		\begin{align}\label{final2}\nonumber
			d(Nu, Nv) & =\|Nu-Nv\|_{L_t^q\left(I ; L_x^r L_y^{\tilde{r}}\right)}\\\nonumber
			&=\left\|\int_0^t e^{i(t-\tau)  (-\Delta)^{\frac{m}{2}}} \left( F_p(u)- F_p(v) \right)d \tau\right\|_{L_t^q\left(\mathcal{I} ; L_x^r L_y^{\tilde{r}}\right)} \\\nonumber
			& \leq C\left\|F_p(u)-F_p(v)\right\|_{L^{q_1^{\prime}}\left(\mathcal{I} ; L_x^{r_1^{\prime}} L_y^{\bar{r}_1^{\prime}}\right)}\\\nonumber
			&\leq C T^{\beta_1(d, p, s)}  \left(\|u\|_{L_t^{q_1}\left(\mathcal{I}; L_x^{r_1} W_y^{s, \tilde{r}_1}\right)}^{p-1}+\|v\|_{L_t^{q_1}\left(\mathcal{I} ; L_x^{r_1} W_y^{s, \tilde{r}_1}\right)}^{p-1}\right)\|u-v\|_{L_t^{q_1}\left(\mathcal{I}; L_x^{r_1} L_y^{\tilde{r}_1}\right)}\\ 				
			& \leq 2 C T^{\beta_1(d, p, s)} A^{p-1} d(u, v)
		\end{align} 	for $u, v \in X(T)$.  By choosing $A=2 C\|f\|_{L_x^2 H_y^s}$ and taking  sufficiently small $T>0$  such  that
		$$
		2 C T^{\beta_1(d, p, s)} A^{p-1} \leq \frac{1}{2}.
		$$
		Then \eqref{final1} and \eqref{final2}	 reduces to 	  
		\begin{align}\label{final11}
			\|Nu\|_{L^q\left(\mathcal{I} ; L_x^r W_y^{s, \bar{r}}\right)} \leq A
		\end{align}
		and  	 \begin{align}\label{final12}
			d(Nu, Nv)  \leq \frac{1}{2} d(u, v),
		\end{align}
		respectively 	for $u, v\in X(T).$  Thus from (\ref{final11}), we see that  $Nu \in X(T)$, i.e.,  $N$ maps $X(T)$ into itself. 
		Also from (\ref{final12}),  it follows that the map $N$ is a contraction mapping on the ball $\mathcal{B}(R)$ around the origin in the Banach space $X(T)$. Therefore, by Banach's fixed point theorem,  there exists a uniquely determined fixed point $u^*$ of the operator $N$, which means $u^*=Nu^* \in X(T)$.  This fixed point $u^*$ will be  our mild solution to (\ref{higher order})    on $[0, T]$. This implies that there exists a local   small data Sobolev solution $u^*$ of the equation $ u^*=Nu^* $ in $ X(T)$, which also gives the solution to the   equation (\ref{higher order})  and this completes the proof of the local well-posedness theorem. 
	\end{proof}

	\section{Strichartz estimates for Dispersive Dunkl semigroup with  partial regular initial data}\label{sec5}
In this section, we recall some basic definitions and important properties of the Dunkl operators and briefly overview the related harmonic analysis to make the paper self-contained. A complete account of harmonic analysis related to Dunkl operators can be found in \cite{ros, ben, Ratna3, Mejjaoli2009, MS, Jiu-Li, FMSW}.  
	\subsection{Dunkl (differential-difference) operators}\label{dunklop} Root systems together with finite reflection groups form the fundamental building blocks of the theory of Dunkl operators. We begin this section by recalling the definition of a root system.
	
	Let $\langle\cdot, \cdot\rangle$  denotes the standard Euclidean scalar product in $\mathbb{R}^{d}$. For $x \in \mathbb{R}^{d}$, we denote   $|x|$ as $|x|=\langle x, x\rangle^{1 / 2}$.
	For $\alpha \in \mathbb{R}^{d} \backslash\{0\},$ denote $r_{\alpha}$ as  the reflection with respect to the hyperplane $\langle \alpha\rangle^{\perp}$, orthogonal to $\alpha$ and is defined by
	$$
	r_{\alpha}(x):=x-2 \frac{\langle \alpha, x\rangle}{|\alpha|^{2}} \alpha, \quad x \in \mathbb{R}^{d}.
	$$
	A finite set $\mathcal{R}$ in $\mathbb{R}^{d} \backslash\{0\}$ is  said to be a  root system if the following holds:
	\begin{enumerate}
		\item $ r_{\alpha}(\mathcal{R})=\mathcal{R}$ for all $\alpha \in \mathcal{R}$,
		\item $\mathcal{R} \cap \mathbb{R} \alpha=\{\pm \alpha\}$ for all $\alpha \in \mathcal{R}$.
	\end{enumerate}
	For a given root system $\mathcal{R}$, the subgroup $G \subset O(n, \mathbb{R})$ generated by the reflections $\left\{r_{\alpha} \mid \alpha \in \mathcal{R}\right\}$ is called the finite Coxeter group associated with $\mathcal{R}$.  The dimension of $span \mathcal{R}$ is called the rank of $\mathcal{R}$.   For a   detailed study on the theory of the   finite reflection groups, we refer the reader to  \cite{hum}.   For some $\beta\in\mathbb{R}^{d}\backslash\bigcup_{\alpha\in\mathcal{R}}\langle \alpha\rangle^{\perp}$, let    $\mathcal{R}^+:=\{\alpha\in\mathcal{R}:\langle\alpha,\beta\rangle>0\}$ be a  fix  positive root system.  
	
	Typical examples of such systems include the Weyl groups, for instance, the symmetric group  $S_{n}$ corresponding to the root system of type  $A_{n-1}$ and the hyperoctahedral group associated with type $B_{n}$. In addition, $H_{3}, H_{4}$ (icosahedral groups) and, $I_{2}(n)$ (symmetry group of the regular $n$-gon)  also belong to the class of Coxeter groups.
	
	A multiplicity function associated with
$G$ is a function $\kappa: \mathcal{R} \rightarrow \mathbb{C}$ that remains constant on the $G$-orbits. Setting $\kappa_{\alpha}:=\kappa(\alpha)$ for $\alpha \in \mathcal{R},$  from the definition of $G$-invariant,  we have $\kappa_{g \alpha}=\kappa_{\alpha}$ for all $g \in G$.  We say $\kappa$ is non-negative if $\kappa_{\alpha} \geq 0$ for all $\alpha \in \mathcal{R}$. The $\mathbb{C}$-vector space of non-negative multiplicity functions on $\mathcal{R}$ is denoted by $\mathcal{K}^{+}$. We also  denote  $\gamma$ as $\gamma=\gamma(\kappa):=\sum\limits_{\alpha\in\mathcal{R}^+}\kappa(\alpha)$. For simplicity of notation, we will use the notation  $N:=d+2\gamma$ to denote the homogeneous dimension throughout the paper. 
	
	In 1989, Dunkl introduced, for $\xi \in \mathbb{C}^{n}$ and $\kappa \in \mathcal{K}^{+},$ a family of first-order differential–difference operators $T_{\xi}:= T_{\xi}(\kappa)$ defined by
	\begin{align}\label{dunkl}
		T_{\xi}(\kappa) f(x):=\partial_{\xi} f(x)+\sum_{\alpha \in \mathcal{R}^{+}} \kappa_{\alpha }\langle \alpha, \xi\rangle \frac{f(x)-f\left(r_{\alpha} x\right)}{\langle \alpha, x\rangle}, \quad f \in C^{1}\left(\mathbb{R}^{d}\right),
	\end{align}
	where $\partial_{\xi}$ denotes the directional derivative corresponding to $\xi$.  The operator $T_\xi$   defined in (\ref{dunkl}) and commonly referred to as the Dunkl operator, represents one of the most significant advances in the theory of special functions linked to root systems \cite{dun}. These operators commute with each other and are skew-symmetric relative to the $G$-invariant measure $h_{\kappa}(x)dx$, where the weight function is given by $$h_{\kappa}(x):=\prod\limits_{\alpha\in \mathcal{R}^{+}}|\langle \alpha, x\rangle|^{2 \kappa_{\alpha}}$$ is of homogeneous degree $2\gamma$. Because the multiplicity function is $G$-invariant  the definition of the Dunkl operator does not depend on the particular choice of the positive subsystem  $\mathcal{R}^{+}$.  In \cite{dun1991}, it is shown that for any $\kappa\in\mathcal{K}^+$, there is a unique linear isomorphism $V_\kappa$ (Dunkl's intertwining operator) on the space $\mathcal{P}(\mathbb{R}^{d})$ of polynomials on $\mathbb{R}^{d}$ such that
	\begin{enumerate}
		\item $V_\kappa\left(\mathcal{P}_m(\mathbb{R}^{d})\right)=\mathcal{P}_m(\mathbb{R}^{d})$ for all $m\in\mathbb{N}$,
		\item $V_\kappa|_{\mathcal{P}_0(\mathbb{R}^{d})}=id$,
		\item $T_\xi(\kappa)V_\kappa=V_\kappa\partial_\xi$,
	\end{enumerate}
	where  $\mathcal{P}_m(\mathbb{R}^{d})$ denotes the space of homogeneous polynomials of degree $m$. R\"{o}sler in \cite{ros} proved that, for any finite reflection group $G$ and   any $k\in\mathcal{K}^+$,   there exists a unique positive Radon probability measure $\rho_x^\kappa$ on $\mathbb{R}^{d}$ such that
	\begin{equation}\label{239}	
    V_\kappa f(x)=\int_{\mathbb{R}^{d}}f(\xi)d\rho_x^\kappa(\xi).
	\end{equation}
The measure $\rho_x^k$ depends on $x\in\mathbb{R}^{d}$ and its support is contained in the ball $B\left(\|x\|\right):=\{\xi\in\mathbb{R}^{d}: \|\xi\|\leq\|x\|\}$. Owing to the Laplace-type representation \eqref{239}, Dunkl’s intertwining operator  $V_\kappa$ admits an extension to a broader class of function spaces.
	
	Let $\{\xi_1, \xi_2, \cdots, \xi_d\}$ be an orthonormal basis of $(\mathbb{R}^{d} , \langle \cdot, \cdot\rangle )$. Then  the Dunkl Laplacian operator $\Delta_\kappa $ is defined as $$\Delta_\kappa=\sum_{j=1}^dT^2_{\xi_j}(\kappa).$$
	The definition of  $\Delta_\kappa$ is independent of the choice of the orthonormal basis of $\mathbb{R}^{d}.$ In fact,   one can see that the operator $\Delta_\kappa $ also can be expressed as $$\Delta_\kappa f(x)=\Delta f(x)+\sum_{\alpha \in \mathcal{R}^+}\kappa_\alpha \left\{\frac{2 \langle \nabla f(x),  \alpha \rangle}{\langle \alpha, x\rangle } -|\alpha|^2 \frac{f(x)-f(r_\alpha x)}{\langle \alpha, x\rangle^2 } \right\}, \quad f\in C^1(\mathbb{R}^{d}),$$
	where $\nabla$ and $\Delta$ are    the usual gradiant  and   Laplacian operator on $\mathbb{R}^{d}$, respectively.  Note  that, for $\kappa \equiv 0$, coincides with the classical Euclidean Laplacian $\Delta$. 
	
	For every $y\in\mathbb{R}^{d}$, the system 
	\begin{equation*}
		\begin{cases}
			T_\xi u(x,y)=\langle y,\xi\rangle\, u(x,y), \quad \xi\in\mathbb{R}^{d},\\
			u(0,y)=1,
		\end{cases}
	\end{equation*}
	admits a unique analytic solution on $\mathbb{R}^{d}$, which we denote as  $E_\kappa(x,y)$, which is generally referred to as the Dunkl kernel. The kernel has a unique holomorphic extension to $\mathbb{C}^n\times \mathbb{C}^n$ and is given by
	\begin{equation*}
		E_\kappa(x,y)=V_\kappa(e^{\langle \cdot,y\rangle})(x)=\int_{\mathbb{R}^{d}}e^{\langle \xi,y\rangle}d\rho_x^\kappa(\xi),\quad \forall x,y\in\mathbb{R}^{d}.
	\end{equation*}
	When $\kappa \equiv 0$, the Dunkl kernel $E_\kappa(x,y)$  reduces to the exponential $e^{\langle x,y\rangle}$. For $x,y\in\mathbb{R}^{d}$ and  $z,w\in\mathbb{C}^n$, the Dunkl kernel satisfies the following properties:
	\begin{enumerate}[(i)]
		\item $E_\kappa(z,w)=E_\kappa(w,z)$ and $E_\kappa(\lambda z,w)=E_\kappa(z,\lambda w),\quad \lambda\in\mathbb{C}$;
		\item  for any $\nu\in\mathbb{N}^n$, we have
		$$|D_z^\nu E_\kappa(x,z)|\leq |x|^{|\nu|}e^{|Re(z)||x|},$$
		where $D_z^\nu=\frac{\partial^{|\nu|}}{\partial {z_1}^{\nu_1}\cdots {z_n}^{\nu_n}}$ and $|\nu|=\nu_1+\cdots+\nu_n$. In particular, for any $x,y\in\mathbb{R}^{d}$, 
		$$|E_\kappa(x,iy)|\leq 1;$$
		\item $E_\kappa(ix,y)=\overline{E_\kappa(-ix,y)}$ and $E_\kappa(gx,gy)=E_\kappa(x,y),\quad g\in G$.
	\end{enumerate}
	
	\subsection{The Dunkl transform} \label{Dunkl transform}
	For $1\leq p<\infty$, let $L_\kappa^p\left(\mathbb{R}^{d}\right)$ be the space of $L^p$-functions on $\mathbb{R}^{d}$ with respect to the weight $h_{\kappa_1}(x)$ with the $L_\kappa^p\left(\mathbb{R}^{d}\right)$-norm
	\begin{equation*}
		\|f\|_{L_\kappa^p\left(\mathbb{R}^{d}\right)}=\left(\int_{\mathbb{R}^{d}}|f(x)|^ph_{\kappa_1}(x)dx\right)^\frac{1}{p},
	\end{equation*}
	and $\|f\|_{L_\kappa^\infty\left(\mathbb{R}^{d}\right)}=ess \sup\limits_{x\in\mathbb{R}^{d}}|f(x)|$.
	The Dunkl transform is defined on $L_\kappa^1\left(\mathbb{R}^{d}\right)$ by
	\begin{equation}\label{Dunkltransform}
		\mathcal{F}_\kappa f(y)=\frac{1}{c_\kappa}\int_{\mathbb{R}^{d}} f(x) E_\kappa(x,-iy)h_{\kappa_1}(x)dx,
	\end{equation}
	where $c_\kappa$ is the Mehta-type constant defined by
	\begin{equation*}
		c_\kappa=\int_{\mathbb{R}^{d}} e^{-\frac{|x|^2}{2}}h_{\kappa_1}(x)dx.
	\end{equation*}
	When $\kappa \equiv 0$,the Dunkl transform reduces to the classical Fourier transform.  With $N=2\gamma+n$, here we list some basic properties of the Dunkl transform.
	\begin{enumerate}[(i)]
		\item  For all $f\in L_\kappa^1\left(\mathbb{R}^{d}\right)$, we have
		\begin{equation*}
			\|\mathcal{F}_\kappa f\|_{L_\kappa^\infty\left(\mathbb{R}^{d}\right)}\leq \frac{1}{c_\kappa} \|f\|_{L_\kappa^1\left(\mathbb{R}^{d}\right)}.
		\end{equation*}
		\item For any function $f$ in the Schwartz space $\mathcal{S}\left(\mathbb{R}^{d}\right)$, we have
		\begin{equation*}
			\mathcal{F}_\kappa(T_\xi f)(y)=i\langle \xi, y\rangle \mathcal{F}_\kappa f(y),\quad y,\xi\in\mathbb{R}^{d}.
		\end{equation*}
		In particular, it follows that 
		\begin{equation*}
			\mathcal{F}_\kappa(\Delta_\kappa f)(y)=-|y|^2 \mathcal{F}_\kappa f(y),\quad y\in\mathbb{R}^{d}.
		\end{equation*}
		\item  For all $f\in L_\kappa^1\left(\mathbb{R}^{d}\right)$, we have
		\begin{equation*}
			\mathcal{F}_\kappa (f(\cdot/\lambda))(y)=\lambda^N\mathcal{F}_\kappa f(\lambda y),\quad y\in\mathbb{R}^{d}, \lambda>0.
		\end{equation*} 
		\item The Dunkl transform $\mathcal{F}_\kappa$ is a homeomorphism of the Schwartz space $\mathcal{S}\left(\mathbb{R}^{d}\right)$ and its inverse is given by $\mathcal{F}_\kappa^{-1}g(x)=\mathcal{F}_\kappa g(-x),$ for all $ g\in \mathcal{S}\left(\mathbb{R}^{d}\right)$. In addition, for all $f\in \mathcal{S}\left(\mathbb{R}^{d}\right)$, it satisfies
		\begin{equation*}
			\int_{\mathbb{R}^{d}} |f(x)|^2 h_{\kappa_1}(x)dx=\int_{\mathbb{R}^{d}} |\mathcal{F}_\kappa f(y)|^2 h_\kappa(y)dy.
		\end{equation*}
		\item For every $f\in L_\kappa^1\left(\mathbb{R}^{d}\right)$ such that $\mathcal{F}_\kappa f\in L_\kappa^1\left(\mathbb{R}^{d}\right)$, we have the inversion formula
		\begin{equation*}
			f(x)=\frac{1}{c_\kappa}\int_{\mathbb{R}^{d}} \mathcal{F}_\kappa f(y) E_\kappa(ix,y)h_\kappa(y)dy,\quad  a.e. \;x\in\mathbb{R}^{d}.
		\end{equation*}
		\item  If $f$ is a radial function in $ L_\kappa^1\left(\mathbb{R}^{d}\right)$ such that $f(x)=\tilde{f}(|x|)$, then
		\begin{equation}\label{radial-Dunkl}
			\mathcal{F}_\kappa f(y)=\frac{1}{\Gamma(N/2)}\int_0^\infty \tilde{f}(r)\frac{J_\frac{d-2}{2}(r|y|)}{\left(r|y|\right)^\frac{d-2}{2}}r^{N-1} dr,
		\end{equation}
		where $J_\nu$ denotes the Bessel function of order $\nu>-\frac{1}{2}$.
		
		\item For $1\leq p\leq \infty$ and $s\in\mathbb{R}$, the homogeneous Dunkl-Sobolev space $\dot{H}_\kappa^{p,s}\left(\mathbb{R}^{d}\right)$ is the set of tempered distributions $u \in \mathcal{S}'\left(\mathbb{R}^{d}\right)$ such that $\mathcal{F}^{-1}_\kappa\left(|\cdot|^s \mathcal{F}_\kappa f\right)\in L_\kappa^p\left(\mathbb{R}^{d}\right)$ and is equiped with the norm 
		\begin{equation*}
			\|u\|_{\dot{H}_\kappa^{p,s}\left(\mathbb{R}^{d}\right)}=\left\|\mathcal{F}^{-1}_\kappa\left(|\cdot|^s \mathcal{F}_\kappa f\right)\right\|_{L_\kappa^p\left(\mathbb{R}^{d}\right)}.
		\end{equation*}
		In particular, when $p=2$, we denote $\dot{H}_\kappa^{p,s}\left(\mathbb{R}^{d}\right)$ by $\dot{H}_\kappa^{s}\left(\mathbb{R}^{d}\right)$ in short and $\dot{H}_\kappa^{0}\left(\mathbb{R}^{d}\right)=L_\kappa^2\left(\mathbb{R}^{d}\right)$.\\
	\item 	For any $1\leq p_1\leq p_2\leq\infty$ and $s_1\leq s_2 \in\mathbb{R}$ such that $s_1-\frac{d}{p_1}=s_2-\frac{d}{p_2}$, we have
		\begin{equation}\label{Sobolev-Embedding}
			\dot{H}_\kappa^{p_1,s_1} (\mathbb{R}^{d})\subseteq \dot{H}_\kappa^{p_2,s_2} (\mathbb{R}^{d}).
		\end{equation}
\end{enumerate}
	\subsection{Dunkl convolution operator}\label{sub5.3} For  given $x\in\mathbb{R}^{d}$, the Dunkl translation operator $f\mapsto \tau_x f$ is defined on $\mathcal{S}\left(\mathbb{R}^{d}\right)$ by
	\begin{equation*}
		\mathcal{F}_\kappa(\tau_x f)(y)=E_\kappa(x,iy)\mathcal{F}_\kappa f(y),\quad y\in\mathbb{R}^{d}.
	\end{equation*}
	Currently, an explicit formula for the Dunkl translation operator is available only in two cases. When $f(x)=\tilde{f}(|x|)$ is a continuous radial function in $L_\kappa^2(\mathbb{R}^{d})$, the Dunkl translation operator is represented by (see \cite{ros, Dai-Wang})
	\begin{equation} \label{translator}
		\tau_xf(y)=V_\kappa[\tilde{f}\left(\sqrt{|y|^2+|x|^2-2\langle y,\cdot\rangle}\right)](x)=\int_{\mathbb{R}^{d}}\tilde{f}\left(\sqrt{|y|^2+|x|^2-2\langle y,\xi\rangle}\right)d\rho^\kappa_x(\xi).
	\end{equation}
The Dunkl translation operator can likewise be extended to the space of tempered distributions $\mathcal{S}'\left(\mathbb{R}^{d}\right)$.Extending key results from classical Fourier analysis to the framework of the Dunkl transform proves to be quite challenging. A major difficulty arises from the fact that the Dunkl translation operator is, in general, not positive. Even $L_\kappa^p(\mathbb{R}^{d})$-boundedness of $\tau_x$ is not established in general. On the other hand, if  $f$ is a radial function in $L_\kappa^p(\mathbb{R}^{d})$, $1\leq p\leq \infty$, the following holds (see \cite{GIT2019, Thangavelu-Xu2005})
	\begin{equation*}
		\|\tau_xf\|_{L_\kappa^p(\mathbb{R}^{d})}\leq \|f\|_{L_\kappa^p(\mathbb{R}^{d})}.
	\end{equation*}
	Using the Dunkl translation operator,  the Dunkl convolution  of functions $f,g\in \mathcal{S}\left(\mathbb{R}^{d}\right)$ is defined as 
	\begin{equation*}
		f*_\kappa g(x)=\int_{\mathbb{R}^{d}}\tau_xf(-y)g(y)h_\kappa(y)dy, \quad x\in\mathbb{R}^{d}.
	\end{equation*}
	The Dunkl convolution enjoys the following properties.
	\begin{enumerate}
		\item  $\mathcal{F}_\kappa(f*_\kappa g)=\mathcal{F}_\kappa (f)\mathcal{F}_\kappa (g)$, $\mathcal{F}^{-1}_\kappa(f*_\kappa g)=\mathcal{F}^{-1}_\kappa (f)\mathcal{F}^{-1}_\kappa (g)$, and $f*_\kappa g=g*_\kappa f$.
		\item Young's inequality: Let $1\leq p,q,r\leq\infty$ such that $1+\frac{1}{r}=\frac{1}{p}+\frac{1}{q}$. If $f\in L^p_\kappa(\mathbb{R}^{d})$ and $g$ is a radial function of $L_\kappa^q(\mathbb{R}^{d})$, then $f*_\kappa g\in L_\kappa^r(\mathbb{R}^{d})$ and we have
		\begin{equation}\label{Young}
			\|f*_\kappa g\|_{L_\kappa^r(\mathbb{R}^{d})}\leq \|f\|_{L_\kappa^p(\mathbb{R}^{d})} \|g\|_{L_\kappa^q(\mathbb{R}^{d})}.
		\end{equation}
	\end{enumerate}
	Since the Dunkl transform of radial functions is explicitly expressed by the integral related to Bessel functions, we first list some properties of Bessel functions.
	
	Let $J_\nu$ be the Bessel function of order $\nu>-\frac{1}{2}$, defined as
	\begin{equation*}
		J_\nu(r)=\frac{(\frac{r}{2})^\nu}{\Gamma(\nu+\frac{1}{2})\pi^{\frac{1}{2}}}\int_{-1}^1e^{ir\tau}(1-\tau^2)^{\nu-\frac{1}{2}}d\tau.
	\end{equation*}
	
	\begin{lem}\cite{LMS}\label{Bessel}
		For $r>0$ and $\nu>-\frac{1}{2}$, we have
		\begin{align}
			&(1)J_\nu(r)\leq C_\nu r^\nu,0<r<1;\label{bessel1}\\
			&(2)\frac{d}{dr}\left(r^{-\nu}J_\nu(r)\right)=-r^{-\nu}J_{\nu+1}(r);\label{bessel2}\\
			&(3)J_\nu(r)\leq C_\nu r^{-\frac{1}{2}}, r\geq1;\label{bessel3}\\
			&(4)J_\nu(r)=\frac{(\frac{r}{2})^\nu}{\Gamma(\nu+\frac{1}{2})\Gamma(\frac{1}{2})} \left[ie^{-ir}\int_0^\infty e^{-rt}(t^2+2it)^{\nu-\frac{1}{2}}dt-ie^{ir}\int_0^\infty e^{-rt}(t^2-2it)^{\nu-\frac{1}{2}}dt\right].\label{bessel4}
		\end{align}
	\end{lem}
	\begin{rem}
		From \eqref{bessel4}, we have the identity
		\begin{equation}\label{Bessel-Fourier}
			\frac{J_{\frac{d-2}{2}}(r)}{r^\frac{d-2}{2}}=C\left(e^{ir}h(r)+e^{-ir}\overline{h(r)}\right),
		\end{equation}
		where $$h(r)=-i\int_0^\infty e^{-rt}(t^2-2it)^\frac{d-3}{2}dt,$$
		and  for any $\beta\in\mathbb{N}$, one can get     
		\begin{equation}\label{h}
			\left|\frac{d^\beta}{dr^\beta}h(r)\right|\leq C_\beta(1+r)^{-\frac{d-1}{2}-\beta},
		\end{equation}
		for all $\beta\in\mathbb{N}$. Thus, from \eqref{h}, for any $s\geq1$ and $\beta\in\mathbb{N}$, we can get 
		\begin{equation}\label{big-s}
			\left|\frac{d^\beta}{dr^\beta}\left(\psi(r)^2h(rs) r^{N-1}\right)\right|\leq C_\beta s^{-\frac{d-1}{2}}.
		\end{equation}
	\end{rem}

	\subsection{Partial regularity framework in Dunkl  analysis}
	Let  $\mathcal{R}_1$ and $\mathcal{R}_2$ be two root systems for   $\mathbb{R}^{d-k}$ and   $\mathbb{R}^k$,  respectively.  Then
	
	$$
	\mathcal{R}:=\mathcal{R}_1 \times(0)_{k} \cup(0)_{d-k} \times \mathcal{R}_2 
	$$
	where $(0)_j=(0,0, \ldots, 0) \in \mathbb{R}^j$, is a root system on $\mathbb{R}^{d}$. Let $\kappa_1: \mathcal{R}_1 \rightarrow \mathbb{Z}_{\geq 0}$ be a multiplicity function for $\mathcal{R}_1$ and $\kappa_2: \mathcal{R}_2 \rightarrow \mathbb{Z}_{\geq 0}$ be a multiplicity function for $\mathcal{R}_2$. Then we  define the multiplicity function $\kappa: \mathcal{R} \rightarrow \mathbb{Z}_{\geq 0}$ by$$
	\begin{array}{ll}
		\kappa(\alpha, 0)=\kappa_1(\alpha) & \text { for all } \alpha \in \mathcal{R}_1 \\
		\kappa(0, \beta)=\kappa_2(\beta) & \text { for all } \beta \in \mathcal{R}_2.
	\end{array}
	$$
	This defines $\kappa$ on all of $\mathcal{R}$. For this choice of the multiplicity function, partial counterparts of the Dunkl objects will actually become coordinate wise product of the corresponding linear objects. In fact, for any $(x, y)\in \mathbb{R}^{d}$, let $h_{\kappa_1}(x)$ and $h_{\kappa_2}(y)$ are the
	weight functions on $\mathcal{R}_1$ and $\mathcal{R}_2$ with homogeneous degree $2\gamma_1$ and  $2\gamma_2$.  Then $h_{\kappa}(x, y)=h_{\kappa_1}(x) h_{\kappa_2}(y)$ is the weight on $\mathcal{R}$  with homogeneous degree $2\gamma=2\gamma_1+2\gamma_2$. Moreover,  from definition we have
	$$
	d \mu_{\kappa}\left(x, y\right)=d \mu_{\kappa_1} (x ) d \mu_{\kappa_2} (y ). 
	$$

	Let $E_{\kappa_1}(x_1, y_1)$ and $E_{\kappa_2}(x_2, y_2)$  denotes  the Dunkl kernel   on $\mathbb{R}^k$ and $\mathbb{R}^{d-k}$ respectively corresponding to the multiplicity function $\kappa_1$ and $\kappa_2$.  Then  the kernel $E_{\kappa}$ on    $\mathbb{R}^{d}$ is given by $$E_{\kappa}\left(\left(x_1, x_2 \right),\left(y_1, y_2 \right)\right)=E_{\kappa_1}\left(x_1, y_1\right) E_{\kappa_2}\left(x_2, y_2\right).$$
	\begin{itemize}
		\item 	For $1\leq p<\infty$, let $L_\kappa^p\left(\mathbb{R}^{k}\times \mathbb{R}^{d-k}\right)$ be the space of $L^p$-functions on $\mathbb{R}^{d}$ with respect to the weight $h_{\kappa}(x,y)$ with the $L_\kappa^p\left(\mathbb{R}^{d}\right)$-norm
		\begin{equation*}
			\|f\|_{L_\kappa^p\left(\mathbb{R}^{d}\right)}=\left(\int_{\mathbb{R}^{k}} \int_{\mathbb{R}^{d-k}}|f(x,y)|^p h_{\kappa_1}(x) h_{\kappa_2}(y)dx  dy\right)^\frac{1}{p}.
		\end{equation*}
		\item 	Similarly, the space  $L_{\kappa_1, x}^p\left(\mathbb{R}^{d-k}\right)$ be the space of $L^p(\mathbb{R}^{k})$ with respect to the weight $h_{\kappa_1}(x)$ with the $L^p(\mathbb{R}^{k})$-norm
		\begin{equation*}
			\|g\|_{L_{\kappa_1, x}^p\left(\mathbb{R}^{k}\right)}=\left(\int_{\mathbb{R}^{k}}|g(x,y)|^p  h_{\kappa_1}(x)  dy\right)^\frac{1}{p}<\infty.
		\end{equation*}

		\item 	Similarly,  $L_{\kappa_2, y}^p\left(\mathbb{R}^{d-k}\right)$ be the space of $L^p(\mathbb{R}^{d-k})$ with respect to the weight $h_{\kappa_2}(y)$ with the $L^p(\mathbb{R}^{d-k})$-norm
		\begin{equation*}
			\|f\|_{L_{\kappa_2, y}^p\left(\mathbb{R}^{d-k}\right)}=\left(\int_{\mathbb{R}^{d-k}}|f(x,y)|^p  h_{\kappa_2}(y)  dy\right)^\frac{1}{p}.
		\end{equation*}
		\item Since the root system $\mathcal{R}$ is the union of two  orthogonal root system, we can  define Dunkl operator on $\mathcal{R}$ coordinate wise as
        \begin{align*}
		T_{j}(\kappa) f(x):=\partial_{j} f(x)+\sum_{\alpha \in \mathcal{R}^{+}} \kappa_{\alpha }\langle \alpha, \xi\rangle \frac{f(x)-f\left(r_{\alpha} x\right)}{\langle \alpha, x\rangle}, \quad f \in C^{1}(\mathbb{R}^{d}).
	\end{align*}
Consequently, the Dunkl Laplacian operator $\Delta_\kappa $ is defined as $$\Delta_\kappa=\sum_{j=1}^dT^2_{j}(\kappa).$$
           \item We define the Dunkl gradient of a complex valued function $f$ as the operator valued operator $$ \nabla_\kappa f=\left( T_1, T_2f, \cdots, T_df \right)$$ 
           and modulus by 
           $$|\nabla_\kappa f|(x)=\left(\sum_{j=1}^d |T_{j}(\kappa)f(x)|^2\right)^{\frac{1}{2}}.$$ 
           \item The homogeneous fractional operator $|\nabla_\kappa|^{s}$ is defined as $$|\nabla_\kappa|^{s}f:=\mathcal{F}^{-1}_\kappa\left(|\cdot|^s \mathcal{F}_\kappa f\right).$$
		\item For $1\leq p\leq \infty$ and $s\in\mathbb{R}$, the homogeneous Dunkl-Sobolev space $\dot{H}_\kappa^{p,s}\left(\mathbb{R}^{d}\right)$ is defined as $$\dot{H}_\kappa^{p,s}\left(\mathbb{R}^{d}\right)=\left\{u \in \mathcal{S}'\left(\mathbb{R}^{d}\right):   	\|u\|_{\dot{H}_\kappa^{p,s}\left(\mathbb{R}^{d}\right)}=\left\|\mathcal{F}^{-1}_\kappa\left(|\cdot|^s \mathcal{F}_\kappa f\right)\right\|_{L_\kappa^p\left(\mathbb{R}^{d}\right)}<\infty\right \}.$$
	\end{itemize}

\subsection{Littlewood-Paley projections}\label{L-P section}
Let $\psi: \mathbb{R}^{d} \rightarrow[0,1]$ be a radial smooth cut-off function supported in $$\{ \xi \in\left.\mathbb{R}^{d}  : \frac{1}{2} \leq|\xi | \leq 2\right\}$$ such that
	$$
	\sum_{j \in \mathbb{Z}} \psi\left(2^{-j} \xi \right)=1.
	$$
   For $j \in \mathbb{Z}$, the Littlewood-Paley operator $\mathcal{P}_j$ is defined as follows:
	$$
	\mathcal{F}_\kappa(\mathcal{P}_jf)(\xi,\eta)=\psi_j(\xi,\eta) \mathcal{F}_\kappa(f)(\xi,\eta),
	$$
	where $\psi_j(\xi,\eta):=\psi\left(2^{-j}\xi,2^{-j}\eta\right)$  and its supporst is the set  $\left\{\xi\in \mathbb{R}^d: 2^{j-1} \leq|(\xi,\eta) | \leq 2^{j+1}\right\}$.The Littlewood-Paley theorem  on mixed Lebesgue spaces $L^r_{{\K_1,x}}L^{\tilde{r}}_{\K_2,y}(\mathbb{R}^{d-k} \times \mathbb{R}^k) $; for $1<r, \tilde{r}<\infty$,
	is given by 	\begin{equation} \label{DLP}
	    \|f\|_{L^r_{{\K_1,x}}L^{\tilde{r}}_{\K_2,y} }\lesssim\left\| \left(\sum_{j \in \mathbb{Z}} \left|\mathcal{P}_j f\right|^2\right)^{1 / 2}\right\|_{L^r_{{\K_1,x}}L^{\tilde{r}}_{\K_2,y}} .
	\end{equation}
	The    Littlewood-Paley operators    $\mathcal{P}_j$ also satisfy the following important properties. 
	\begin{itemize}
		\item Identity relation: $f= \displaystyle \sum_{j \in \mathbb{Z}} \mathcal{P}_j f.$
		\item  Quasi-orthogonal relations: $\mathcal{P}_j \mathcal{P}_k=0$ if $|j-k|>1.$
		\item There exists a constant $C>0$ such that $$C^{-1} \sum_{j \in \mathbb{Z}}\left\| \mathcal{P}_j f\right\|_{L_\kappa^2\left(\mathbb{R}^{d}\right)}^2 \leq\|f\|_{L_\kappa^2\left(\mathbb{R}^{d}\right)}^2 \leq C \sum_{j \in \mathbb{Z}}\left\|\mathcal{P}_jf\right\|_{L_\kappa^2\left(\mathbb{R}^{d}\right)}^2.$$
		\item  Sobolev norm: $$\|f\|_{H_\kappa^s\left(\mathbb{R}^{d}\right)}\sim   \left\|(2^{js}\left\| \mathcal{P}_j f\right\|_{L_\kappa^2\left(\mathbb{R}^{d}\right)})\right \|_{\ell^2(\mathbb{Z})} .$$
		Moreover, there exists a constant $C>0$ such that
		$$C^{-1} \sum_{j \in \mathbb{Z}} 2^{2js}\left\| \mathcal{P}_jf\right\|_{L_\kappa^2\left(\mathbb{R}^{d}\right)}^2 \leq\|f\|_{H_\kappa^s\left(\mathbb{R}^{d}\right)}^2 \leq C \sum_{j \in \mathbb{Z}}2^{2js}\left\|\mathcal{P}_j f\right\|_{L_\kappa^2\left(\mathbb{R}^{d}\right)}^2.$$
	\end{itemize}
Also, in the Dunkl partial frame,  all the properties mentioned in Subsections \ref{Dunkl transform} and \ref{sub5.3} holds.

	\subsection{Strichartz estimates for higher order Dunkl-Sch\"rodinger semigroups}\label{sec6}
	Consider the nonlinear Schr\"odinger and wave equations
	\begin{equation}\label{DNLS}
		\begin{cases}
			i\partial_t u +(-\Delta_\kappa)^{\frac{m}{2}} u= F_p(u), \\
			u(0, x)=f(x),
		\end{cases}
	\end{equation}
	where $m>0,$ $(t, x) \in \mathbb{R} \times \mathbb{R}^{d}$ and the nonlinearity $F_p \in C^1$ with $p>1$ satisfies same condition as in \eqref{as1}.
If $u(t,x)$ is a solution of \eqref{DNLS}, then so is
	\begin{equation*}
		u_{\delta}(t, x)= \delta^{\frac{m}{p-1}} u(\delta^{m} t, \delta x), \quad \delta>0.
	\end{equation*}
	In addition, the Sobolev norm of the rescaled initial data $f_{\delta}(x)=u_{\delta}(0, x)$ is given in terms of the original $f$ as 
	\begin{equation}\label{scalingD}
		\| f_{\delta}\|_{\dot{H}_\kappa^s(\mathbb{R}^{d})} =\delta^{\frac{m}{p-1}+s-\frac{N}{2}} \|f\|_{ \dot{H}^s(\mathbb{R}^{d})}.
	\end{equation}	On the other hand 
	$$\| f_{\delta}\|_{L_{\kappa_1, x}^2 \dot{H}_{\kappa_2, y}^s} =\delta^{\frac{m}{p-1}+s-\frac{d+2(\gamma_1+\gamma_2)}{2}} \|f\|_{L_{\kappa_1, x}^2 \dot{H}_{\kappa_2, y}^s}=\delta^{\frac{m}{p-1}+s-\frac{N }{2}} \|f\|_{L_{\kappa_1, x}^2 \dot{H}_{\kappa_2, y}^s},$$  where $ \dot{H}_{\kappa_2, y}^s$ denotes the partial Dunkl-Sobolev space with respect to the $y$ variable. 
	Therefore, the same range of $p$ is guaranteed even when the regularity is imposed only partially, which naturally suggests the possibility of developing a new well-posedness theory under weaker regularity assumptions. The first step in this direction is to study Strichartz estimates for the free Dunkl–Schrödinger propagator with partially regular initial data. The following theorem, which constitutes the main result of this section, addresses this issue.
	\begin{theorem}\label{D Main Dunkl}
		Let $d\geq 1,$ $1 \leq k \leq d,$ $  2 \leq \widetilde{r} \leq r<\infty ,m\in (0, \infty)\backslash \{1\}$ and $ 2<q \leq \infty$. If $ q, r, \tilde{r}$ satisfy
		\begin{align}\label{Deq0}
			\frac{2}{q} \leq(d+2\gamma_1-k)\left(\frac{1}{2}-\frac{1}{r}\right)+(k+2\gamma_2)\left(\frac{1}{2}-\frac{1}{\widetilde{r}}\right). 
		\end{align}
		Then we have
		\begin{align}\label{Deq1}
			\left\|e^{i t(-\Delta_\kappa)^{\frac{m}{2}}} f\right\|_{L_t^q\left(\mathbb{R} ; L_{{\K_1,x}}^r\left(\mathbb{R}^{d-k} ; L_{{\K_2,y}}^{\tilde{r}}\left(\mathbb{R}^k\right)\right)\right)} \lesssim\|f\|_{\dot{H}_{\kappa}^s}^2,
		\end{align}
		under the scaling condition
		\begin{align}\label{Deq2}
			\frac{{{m}}}{q}=-s +(d+2\gamma_1-k)\left(\frac{1}{2}-\frac{1}{r}\right)+(k+2\gamma_2)\left(\frac{1}{2}-\frac{1}{\widetilde{r}}\right).
		\end{align}
	\end{theorem}

    Now, we obtain the following result, which can be applied in various contexts.

	\begin{lem}\label{general case}
		  Let  $\psi$ be any smooth function on $\mathbb{R}$ supported in $[\frac{1}{2},2]$ and let $\phi: \mathbb{R}^+ \to \mathbb{R}$ be a smooth function.
			\begin{itemize}
				\item If $\phi'(r)\neq0,$ $\forall r\in [\frac{1}{2},2]$, for $(x, t) \in \mathbb{R}^{d+1}$, we have
				$$
				\left|\int_{\mathbb{R}^{d}} E_{\kappa}(ix, \xi) e^{-i t\phi(|\xi|)} \psi(|\xi|) h_\kappa(\xi)  d \xi\right| \lesssim_{\phi,\psi}
				(1+|(x, t)|)^{-\frac{N-1}{2}}.
				$$
				\item If $\phi'(r),\phi''(r)\neq0,$ $\forall r\in [\frac{1}{2},2]$, for $(x, t) \in \mathbb{R}^{d+1}$, we have
				$$
				\left|\int_{\mathbb{R}^{d}} E_{\kappa}(ix, \xi) e^{-i t\phi(|\xi|)} \psi(|\xi|) h_\kappa(\xi)  d \xi\right| \lesssim_{\phi,\psi}
				(1+|(x, t)|)^{-\frac{N}{2}}.
				$$
			\end{itemize}
	\end{lem}
	\begin{proof}
		Let us denote
		\begin{align*}
			I_\kappa(t, x):&=\int_{\mathbb{R}^{d}} E_{\kappa}(ix, \xi) e^{-i t\phi(|\xi|)} \psi(|\xi|) h_\kappa(\xi)  d \xi.
		\end{align*}
		Then it is easy to see that
		\begin{equation*}
			\left|I_\kappa(t, x)\right|\leq \int_{\mathbb{R}^{d}} |\psi(|\xi|)| h_\kappa(\xi)  d \xi=\|\psi\|_{L^1_\kappa(\mathbb{R}^{d})}\lesssim_\psi 1.
		\end{equation*}
		Hence, it suffices to prove the following statement: 
        \begin{itemize}
        \item If $\phi'(r)\neq0$ for all $r\in [\frac{1}{2},2]$, we have
		\begin{equation}\label{N-1}
			\left|I_\kappa(t, x)\right|\lesssim_{\phi,\psi} |(x, t)|^{-\frac{N-1}{2}};	\end{equation}
            \item In addition, if $\phi''(r)\neq0$ for all $r\in [\frac{1}{2},2]$, we have
		\begin{equation}\label{N}
			\left|I_\kappa(t, x)\right|\lesssim_{\phi,\psi} |(x, t)|^{-\frac{N}{2}}.
		\end{equation}
        \end{itemize} 
Now from \eqref{radial-Dunkl} and \eqref{Bessel-Fourier}, we see that 
		\begin{align*}
			I_\kappa(t, x)&=\frac{1}{\Gamma(N/2)}\int_0^\infty e^{-it\phi(r)}\psi(r)\frac{J_\frac{N-2}{2}(r|x|)}{\left(r|x|\right)^\frac{N-2}{2}}r^{N-1} dr\\
			&= C \int_0^\infty \left(e^{i(-t\phi(r)+r|x|)}h(r|x|)+e^{i(-t\phi(r)-r|x|)}\overline{h(r|x|)}\right)\psi(r)r^{N-1} dr\\
			&= C \sum_{\pm}\int_0^\infty e^{i\left(\pm r|x|-t\phi(r)\right)}\psi(r)H(r|x|)r^{N-1}dr.
		\end{align*}
		where $H=h$ or $\overline{h}$. We divide our discussion in two cases.
		\begin{itemize}
			\item \textbf{Case 1:} $|x|\geq 2\max\limits_{r\in[1/2,2]}|\phi'(r)||t|$ or $|x|\leq \frac{1}{2}\min\limits_{r\in[1/2,2]}|\phi'(r)||t|$. In this case, we have
			\begin{equation*}
				\left|\frac{d}{dr}\left(\pm r|x|-t\phi(r)\right)\right|=\left|\pm |x|-t\phi'(r)\right|\gtrsim_{\phi} |(x,t)|.
			\end{equation*}
			For any $M\in\mathbb{N}$, by repeated integration by parts $M$-times and \eqref{big-s}, we obtain 
			\begin{equation*}
				\left|\int_0^\infty e^{i\left(\pm r|x|-t\phi(r)\right)}\psi(r)H(r|x|)r^{N-1}dr\right|\lesssim_{\phi,\psi,M} \frac{1}{|(x,t)|^M}.
			\end{equation*}
			Therefore
			\begin{equation}\label{N-1 1}
				\left|I_\kappa(t, x)\right|\lesssim_{\phi,\psi,M} \frac{1}{|(x,t)|^M}.
			\end{equation}
			\item \textbf{Case 2:}  $\frac{1}{2}\min\limits_{r\in[1/2,2]}|\phi'(r)||t|\leq |x|\leq 2\max\limits_{r\in[1/2,2]}|\phi'(r)||t|$, i.e., $|x|\sim|t|$. We see that 
			\begin{equation*}
				I_\kappa(t, x)=\mathcal{F}^{-1}_\kappa (e^{-it\phi} \psi(|\cdot|))(x).
			\end{equation*}
			From the proof of Theorem 1.1 in \cite{LMS}, we can get
			\begin{equation}\label{N-1 2}
				\left|I_\kappa(t, x)\right|=\|\mathcal{F}^{-1}_\kappa (e^{-it\phi} \psi(|\cdot|))\|_{L_\kappa^\infty(\mathbb{R}^{d})} \lesssim_{\psi} \frac{1}{|t|^\frac{N-1}{2}}\lesssim_{\psi}  \frac{1}{|(x,t)|^\frac{N-1}{2}}.
			\end{equation}
			
		\end{itemize}
        Finally, combining \eqref{N-1 1} and \eqref{N-1 2}, we prove \eqref{N-1}. Moreover, if $\phi''(r)\neq0$ for all $r\in [\frac{1}{2},2]$, from the proof of Theorem 1.1 in \cite{LMS}, using Van der Corput lemma, when $\frac{1}{2}\min\limits_{r\in[1/2,2]}|\phi'(r)||t|\leq |x|\leq 2\max\limits_{r\in[1/2,2]}|\phi'(r)||t|$, we have
		\begin{equation}\label{N 2}
			\left|I_\kappa(t, x)\right|=\|\mathcal{F}^{-1}_\kappa (e^{-it\phi} \psi(|\cdot|))\|_{L_\kappa^\infty(\mathbb{R}^{d})} \lesssim_{\phi,\psi} \frac{1}{|t|^\frac{N}{2}}\lesssim_{\phi,\psi}  \frac{1}{|(x,t)|^\frac{N}{2}}.
		\end{equation}
		Combining \eqref{N-1 1} and \eqref{N 2}, we prove \eqref{N}  and this completes the proof of the lemma.
	\end{proof}

As an consequence of above result we get the following lemma, which will serve as a useful tool in estimating the certain integrals in our main result.

\begin{cor}[Fractional Schr\"odinger case]\label{fractional}
			Assume $\alpha>0$. Let  $\psi$ be any smooth function on $\mathbb{R}$ supported in $[\frac{1}{2},2]$. Then, for $(x, t) \in \mathbb{R}^{d+1}$
			$$
			\left|\int_{\mathbb{R}^{d}} E_{\kappa}(ix, \xi) e^{-i t|\xi|^\alpha} \psi(|\xi|) h_\kappa(\xi)  d \xi\right| \lesssim_\psi\begin{cases}(1+|(x, t)|)^{-\frac{N}{2}},\alpha\neq1,\\
				(1+|(x, t)|)^{-\frac{N-1}{2}},\alpha=1.
			\end{cases}
			$$
	\end{cor}

 To establish the above result, we first prove the following lemmas, which are crucial for extending localized frequency estimates to the global setting.
 \begin{lem}\label{Deq6}
		Let $d \geq 1 $, $1 \leq k \leq d$ and $m\in (0, \infty)\backslash \{1\}$.  Assume that $ \leq \widetilde{r} \leq r \leq \infty$. Then for $(x, y) \in \mathbb{R}^{d-k} \times \mathbb{R}^k$ we have
		\begin{align}\label{Deq5}
			\left\|e^{i t(-\Delta_\kappa)^{\frac{m}{2}}} \mathcal{P}_0 f\right\|_{L_{\kappa_1, x}^r L_{\kappa_2, y}^{\tilde{r}}} \lesssim(1+|t|)^{-\beta(r, \widetilde{r})}\|f\|_{L_{\kappa_1, x}^{r^{\prime}} L_{\kappa_2, y}^{\tilde{r}^{\prime}}}
		\end{align}
		where $$\beta(r, \widetilde{r})= (d+2\gamma_1-k) \Big(\frac{1}{2}-\frac{1}{r}\Big) + (k+2\gamma_2)\Big(\frac{1}{2}-\frac{1}{\widetilde{r}}\Big).$$
	\end{lem} 
	\begin{proof}
		By the Riesz-Thorin interpolation theorem, we only need to obtain (\ref{Deq5}) for the  three cases:
		(a) $r=\widetilde{r}=2$,
		(b) $r=\widetilde{r}=\infty$,
		(c) $r=\infty$ and $\tilde{r}=2$.

		The case (a) directly follows from the Plancherel theorem. For the case (b), i.e., when  $r=\widetilde{r}=\infty$,  considering the Dunkl transform and its inversion formula, we can write
		$$
		\begin{aligned}
			e^{-i t \Delta_\kappa} P_0 f(x, y) 
			& =K*_{\kappa} f(x, y),
		\end{aligned}
		$$
		where the kernel $K$ is given by 
		\begin{align}\label{Deq9}
			K(x, y, t)=\frac{1}{{c_{\K_1}c_{\K_2}}} \int_{\mathbb{R}^{d-k}} \int_{\mathbb{R}^k} E_{\kappa}(i(x,y), (\xi, \eta))e^{ it (|\xi|^2+ |\eta|^2)^{\frac{m}{2}}} \psi(\xi, \eta)h_{\kappa_1}(\xi) h_{\kappa_2}(\eta) d \eta d \xi.
		\end{align}
		Using  Young's inequality, we get
		\begin{align}\label{DKernel estimate 1}
			\left\|e^{-i t\Delta_\kappa} \mathcal{P}_0f\right\|_{L_{\kappa}^{\infty}}=\left\|K*_{\kappa} f\right\|_{L_{\kappa}^{\infty}} \leq\left\|K\right\|_{L_{\kappa}^{\infty}}\|f\|_{L_{\kappa}^1} .
		\end{align}
		Now, our aim is to estimate the kernel $K$.  By  Corollary \ref{fractional},    (\ref{Deq9})  can be estimated in the following way:
		\begin{align}\label{oscillatory}\nonumber
			\left|K(x, y, t)\right| & =\frac{1}{{ c_{\K_1}c_{\K_2}}} \left|\int_{\mathbb{R}^{d-k}} \int_{\mathbb{R}^k} E_{\kappa}(i(x,y), (\xi, \eta))e^{i t (|\xi|^2+ |\eta|^2)^{\frac{m}{2}}} \psi(\xi, \eta)h_{\kappa_1}(\xi) h_{\kappa_2}(\eta)  d \eta d \xi \right| \\\nonumber
			& \lesssim(1+|(x, y, t)|)^{-\frac{N}{2}} \\
			& \lesssim(1+|t|)^{-\frac{N}{2}}.
		\end{align}
		Then from the inequality \eqref{DKernel estimate 1}, we get 
		$$		\left\|e^{i t(-\Delta_\kappa)^{\frac{m}{2}}} \mathcal{P}_0 f\right\|_{L_{\kappa}^{\infty}}=\left\|K*_{\kappa} f\right\|_{L_{\kappa}^{\infty}}   \lesssim(1+|t|)^{-\beta(\infty, \infty)} \|f\|_{L_{\kappa}^1} $$
		with $ \beta(\infty, \infty)=N/ 2 \geq 0$ and   this concludes part (b). 
		
		For the case $(c)$ i.e., when $r=\infty$ and $\tilde{r}=2$,  using Minkowski's inequality and Plancherel's theorem with respect to $y$ variable, we get 
		\begin{align}\label{DKernel 2}\nonumber
			\left\|e^{i t(-\Delta_\kappa)^{\frac{m}{2}}}\mathcal{P}_0 f\right\|_{L_{\kappa_1, x}^{\infty} L_{\kappa_2, y}^2} 
			& =\left\|\| K*_{\kappa_2} f \|_{L_{\kappa_2, y}^2}\right\|_{L_{\kappa_1, x}^{\infty}} \\\nonumber
			& \leq\left\|\int_{\mathbb{R}^{d-k}} \| K\left(x-x^{\prime}, \cdot\right) *_{\kappa_2, y} f\left(x^{\prime}, \cdot\right) \|_{L_{\kappa_2, y}^2} h_{\kappa_1}(x') d x^{\prime}\right\|_{L_{\kappa_1, x}^{\infty}} \\
			& =\left\|\int_{\mathbb{R}^{d-k}} \| \widetilde{K}\left(x-x^{\prime}, \cdot\right) \tilde{f}\left(x^{\prime}, \cdot\right) \|_{L_{\kappa_2, \eta}^2} h_{\kappa_1}(x')  d x^{\prime}\right\|_{L_{\kappa_1, x}^{\infty}},
		\end{align}
		where  $\tilde{f}=\mathcal{F}_{{ \K_2,y}}(f(x, \cdot))$ denotes the  Dunkl transform of $f$ with respect to  the   $y \in \mathbb{R}^k$ variable and similarly $\tilde{K}=\mathcal{F}_{{\K_2,y}}(K(x, \cdot))$, which is given by 
		$$
		\widetilde{K}(x, \eta, t)=\frac{1}{{c_{\K_1}}} \int_{\mathbb{R}^{d-k}}  E_{\kappa}(i (x, y), (\xi, \eta))e^{i t (|\xi|^2+ |\eta|^2)^{\frac{m}{2}}} \psi(\xi, \eta)h_{\kappa_1}(\xi)    d \xi.
		$$
Again,  applying Corollary \ref{fractional}  for fixed $|\eta| \leq 2$, we  get
		\begin{align}\label{Deq7}
			\sup _{|\eta| \leq 2}\left|\widetilde{K}(x, \eta, t)\right| \lesssim(1+|(x, t)|)^{-\frac{d+2\gamma_1-k}{2}},
		\end{align}
        where $(d-k)+2\gamma_1$ is the homogeneous dimension of $\mathbb{R}^{d-k}.$ 
        
		Using  the estimate  \eqref{Deq7} and   Young's inequality, from \eqref{DKernel 2},  we obtain 
		$$
		\begin{aligned}
			\left\|e^{i t(-\Delta_\kappa)^{\frac{m}{2}}}\mathcal{P}_0 f\right\|_{L_{\kappa_1, x}^{\infty} L_{\kappa_2, y}^2} 
			&\leq	\left\|\int_{\mathbb{R}^{d-k}} \| \widetilde{K}\left(x-x^{\prime}, \cdot\right) \tilde{f}\left(x^{\prime}, \cdot\right) \|_{L_{\kappa_2, \eta}^2}h_{\kappa_1}(x')   d x^{\prime}\right\|_{L_{\kappa_1, x}^{\infty}}\\
			&	\lesssim 	\left\|\int_{\mathbb{R}^{d-k}}  (1+|(x-x', t)|)^{-\frac{d+2\gamma_1-k}{2}}\|   \tilde{f}\left(x^{\prime}, \cdot\right) \|_{L_{\kappa_2, \eta}^2} h_{\kappa_1}(x')   d x^{\prime}\right\|_{L_{\kappa_1, x}^{\infty}}\\ & \lesssim\left\|(1+|(\cdot, t)|)^{-\frac{d+2\gamma_1-k}{2}} *_{\kappa,x} \| \tilde{f} \|_{L_{\kappa_2, \eta}^2}\right\|_{L_{\kappa_1, x}^{\infty}} \\
			& \lesssim(1+|t|)^{-\frac{d+2\gamma_1-k}{2}}\left \| \| \tilde{f} \|_{L_{\kappa_2, \eta}^2}\right\|_{L_{\kappa_1, x}^1} \\
			& \leq(1+|t|)^{-\frac{d+2\gamma_1-k}{2}}\|f\|_{L_{\kappa_1, x}^1 L_{\kappa_2, y}^2}\\
			& =(1+|t|)^{-\beta(\infty, 2)}\|f\|_{L_{\kappa_1, x}^1 L_{\kappa_2, y}^2} 
		\end{aligned}
		$$
		with $ \beta(\infty, 2)=\frac{d+2\gamma_1-k}{2} \geq 0$. This completes the proof of the lemma.  
	\end{proof}
 	\begin{lem} 	Let $d \geq 1,1 \leq k \leq d$,    $m\in (0, \infty)\backslash \{1\}$, and 
		\begin{align*}
			\frac{2}{q} \leq(d+2\gamma_1-k)\left(\frac{1}{2}-\frac{1}{r}\right)+(k+2\gamma_2)\left(\frac{1}{2}-\frac{1}{\widetilde{r}}\right), \quad 2 \leq \widetilde{r} \leq r<\infty, \quad 2\leq q \leq \infty, 
		\end{align*}
		then  
		\begin{align}\label{Deq4}
			\left\|e^{i t(-\Delta_\kappa)^{\frac{m}{2}} } \mathcal{P}_0 f\right\|_{L_t^q L_{\kappa_1, x}^r L_{\kappa_2, y}^{\tilde{\tau}}} \lesssim\|f\|_{L_{\kappa}^2}.
		\end{align}

	\end{lem}
	\begin{proof}
		By   the standard $T T^*$ argument,  the required estimate (\ref{Deq4})  is equivalent to
		$$
		\left\|\int_{\mathbb{R}} e^{i(t-\tau)(-\Delta_\kappa)^{\frac{m}{2}}} \mathcal{P}_0  g d \tau\right\|_{L_t^q L_{\kappa_1, x}^r L_{\kappa_2, y}^{\tilde{r}}} \lesssim\|g\|_{L_t^{q^{\prime}} L_{\kappa_1, x}^{r^{\prime}} L_{\kappa_2, y}^{\tilde{r}^{\prime}}}.
		$$
		Now  by  the estimate \eqref{Deq5} in Lemma \ref{Deq6}, we have
		$$
		\begin{aligned}
			\left\|\int_{\mathbb{R}} e^{i(t-\tau)(-\Delta_\kappa)^{\frac{m}{2}}} \mathcal{P}_0 g d \tau\right\|_{L_t^q L_{\kappa_1, x}^r L_{\kappa_2, y}^{\tilde{r}}} & \leq\left\|\int_{\mathbb{R}} \| e^{i(t-\tau)(-\Delta_\kappa)^{\frac{m}{2}}} \mathcal{P}_0 g \|_{L_{\kappa_1, x}^r L_{\kappa_2, y}^{\tilde{r}}} d \tau\right\|_{L_t^q} \\
			& \lesssim\left\|\int_{\mathbb{R}}  (1+|t-\tau|)^{-\beta(r, \widetilde{r})}\|g\|_{L_{\kappa_1, x}^{r^{\prime}} L_{\kappa_2, y}^{\tilde{r}^{\prime}}} d \tau\right\|_{L_t^q} \\
			& \lesssim\left\|(1+|\cdot|)^{-\beta(r, \widetilde{r})} *_t \| g \|_{L_{\kappa_1, x}^{r^{\prime}} L_{\kappa_2, y}^{\tilde{r}^{\prime}}}\right\|_{L_t^q}.
		\end{aligned}
		$$
		From the notation, $	\frac{2}{q} \leq(d+2\gamma_1-k)\left(\frac{1}{2}-\frac{1}{r}\right)+(k+2\gamma_2)\left(\frac{1}{2}-\frac{1}{\widetilde{r}}\right)$  reduces to   $2 / q \leq \beta(r, \tilde{r})$. We now consider two cases.
		\begin{itemize}
			\item When $2 / q<\beta(r, \tilde{r})$ and $2 \leq q \leq \infty$:    Young's inequality yields
			$$
			\begin{aligned}
				\left\|(1+|\cdot|)^{-\beta(r, \widetilde{r})} *_t \| g \|_{L_{\kappa_1, x}^{r^{\prime}} L_{\kappa_2, y}^{\widetilde{r}}}\right\|_{L_t^q} & \lesssim\left\|(1+|\cdot|)^{-\beta(r, \widetilde{r})}\right\|_{L_t^q}\|g\|_{L_t^{q^{\prime}} L_{\kappa_1, x}^{r^{\prime}} L_{\kappa_2, y}^{\widetilde{r}}} \\
				& \lesssim\|g\|_{L_t^{q^{\prime}} L_{\kappa_1, x}^{r^{\prime}} L_{\kappa_2, y}^{\widetilde{r}^{\prime}}}.
			\end{aligned}
			$$
			\item 	When $2 / q=\beta(r, \tilde{r})$ with $2<q<\infty$: applying  the Hardy-Littlewood-Sobolev inequality for one-dimension, to get
			$$
			\begin{aligned}
				\left\|(1+|\cdot|)^{- \beta(r, \tilde{r})} *_t \| g \|_{L_{\kappa_1, x}^{r^{\prime}} L_{\kappa_2, y}^{\tilde{r}^{\prime}}}\right\|_{L_t^q} & \lesssim\left\||\cdot|^{- \beta(r, \tilde{r})} *_t \| g \|_{L_{\kappa_1, x}^{r^{\prime}} L_{\kappa_2, y}^{\tilde{r}^{\prime}}}\right\|_{L_t^q} \\
				& \lesssim\|g\|_{L_t^{q^{\prime}} L_{\kappa_1, x}^{\tau_x^{\prime}} L_{\kappa_2, y}^{\tilde{r}^{\prime}}}.
			\end{aligned}
			$$
			\item 	When  $q=\infty$ and $ \beta(r, \widetilde{r})=0$: this implies  that  $r=\widetilde{r}=2$. In this case, the required estimate (\ref{Deq5})   holds trivially by Plancherel's theorem and hence completes the proof of the lemma.
		\end{itemize}
	\end{proof} 
	Now we are in a position to prove  Strichartz estimates for the higher order Dunkl-Schr\"odinger equation in the partial regularity framework.
	
	\begin{proof}[Proof of  Theorem \ref{D Main Dunkl}]
		Susbtituting $e^{i t(-\Delta_\kappa)^{\frac{m}{2}}} f$  into \eqref{DLP} and  using   Minkowski inequality along with  the that fact that $P_j$ can commute with $e^{i t(-\Delta_\kappa)^{\frac{m}{2}}}$,  we conclude that 
		\begin{align}\label{Deq12}\nonumber
			\left\|e^{i t(-\Delta_\kappa)^{\frac{m}{2}}} f\right\|_{L_t^q L_{\kappa_1, x}^r L_{\kappa_2, y}^{\tilde{r}}}^2 &\lesssim\left\| \left( \sum_{j \in \mathbb{Z}} \left|\mathcal{P}_j e^{i t(-\Delta_\kappa)^{\frac{m}{2}}} f\right|^2\right)^{1 / 2}\right\|_{L_t^q L_{\kappa_1, x}^r L_{\kappa_2, y}^{\tilde{r}}}^2 \\
			& \lesssim \sum_{j \in \mathbb{Z}}\left\|e^{i t(-\Delta_\kappa)^{\frac{m}{2}}} \mathcal{P}_j f\right\|_{L_t^q L_{\kappa_1, x}^r L_{\kappa_2, y}^{\tilde{r}}}^2.
		\end{align}
		Using the Dunkl transform and its  inversion formula, we can write 
		$$
		\begin{aligned}
			&e^{i t(-\Delta_\kappa)^{\frac{m}{2}}} \mathcal{P}_j  f(x, y) \\& =\frac{1}{(2 \pi)^d} \int_{\mathbb{R}^{d}}      E_{\kappa}(i(x, y),(\xi, \eta) ) e^{it  (|\xi|^2+|\eta|^2)^{\frac{m}{2}}} \psi_j(\xi, \eta) \mathcal{F}_\kappa{f} ( \xi,  \eta )   h_{\kappa_2}(\eta )   h_{\kappa_1}(\xi) d \xi  d \eta\\
			&=\frac{1}{(2 \pi)^d}  \int_{\mathbb{R}^{d}}      E_{\kappa}(i(x, y),(\xi, \eta) ) e^{it  (|\xi|^2+|\eta|^2)^{\frac{m}{2}}}   \psi(2^{-j}\xi, 2^{-j}\eta) \mathcal{F}_\kappa{f} ( \xi,  \eta )   h_{\kappa_2}(\eta )   h_{\kappa_1}(\xi) d \xi  d \eta.
		\end{aligned}
		$$
		By the  change of variables $2^{-j} \xi \rightarrow \xi$ and $2^{-j} \eta \rightarrow \eta$,  we get 
		$$
		\begin{aligned}
			&	e^{i t(-\Delta_\kappa)^{\frac{m}{2}}} \mathcal{P}_j  f(x, y) \\&= \frac{2^{jN}}{(2 \pi)^d} \int_{\mathbb{R}^{d}}      E_{\kappa}(i(2^jx, 2^jy),(\xi, \eta) ) e^{it (|2^j\xi|^2+|2^j\eta|^2)^{\frac{m}{2}}}   \psi(\xi, \eta) \mathcal{F}_\kappa{f} ( 2^j\xi, 2^j \eta )   h_{\kappa_2}(2^j\eta )   h_{\kappa_1}(2^j\xi) d \xi  d \eta\\
			& = \frac{2^{jN}}{(2 \pi)^d} \int_{\mathbb{R}^{d}}      E_{\kappa}(i(2^jx, 2^jy),(\xi, \eta) ) e^{i2^{mj} t (|2^j\xi|^2+|2^j\eta|^2)^{\frac{m}{2}}}   \psi(\xi, \eta) \mathcal{F}_\kappa{f} ( 2^j\xi, 2^j \eta )   h_{\kappa_2}(\eta )   h_{\kappa_1}(\xi) d \xi  d \eta\\
			&
			=e^{i 2^{mj} t(-\Delta_\kappa)^{\frac{m}{2}}} \mathcal{P}_0 f_j\left(2^j x, 2^j y\right)
		\end{aligned}
		$$
		where $f_j(x, y):=f\left(2^{-j} x, 2^{-j} y\right)$.  Using (\ref{Deq4}) and the above scaling estimates,  we can write 
		\begin{align}\label{eq14}\nonumber
			\left\|e^{i t(-\Delta_\kappa)^{\frac{m}{2}}} \mathcal{P}_j  f\right\|_{L_t^q L_{\kappa_1, x}^r L_{\kappa_2, y}^{\tilde{r}}}^2 
			&   \lesssim 2^{j\left(-\frac{{m}}{q}-\frac{d+2\gamma_1-k}{r}-\frac{k+2\gamma_2}{\tilde{r}}\right)}\left\|e^{i t(-\Delta_\kappa)^{\frac{m}{2}}} \mathcal{P}_0 f_j\right\|_{L_t^q L_{\kappa_1, x}^r L_{\kappa_2, y}^{\tilde{r}}} \\\nonumber
			& \lesssim 2^{j\left(-\frac{{{m}}}{q}-\frac{d+2\gamma_1-k}{r}-\frac{k+2\gamma_2}{\tilde{r}}\right)}\left\|f_j\right\|_{L_{\kappa}^2} \\\nonumber
			& \lesssim  2^{-\frac{jN}{2}}2^{js}2^{j\left(-s-\frac{{{m}}}{q}-\frac{d+2\gamma_1-k}{r}-\frac{k+2\gamma_2}{\tilde{r}}+\frac{d}{2}\right)}\left\|f_j\right\|_{L_{\kappa}^2} \\
			& \leq  2^{j s}\|f\|_{L_{\kappa}^2},
		\end{align}
		where in the last line we used the scaling condition \eqref{Deq2}.  Using the 		quasi-orthogonal relations $P_jP_k=0$ if $|j-k|>1$ and the identity relation,   if we write $P_j f=P_j \widetilde{P}_j f$   with $\widetilde{P}_j=\sum_{k:|j-k| \leq 1} P_k$, then   from    (\ref{Deq12}) and \eqref{eq14}, we can write  
		$$
		\begin{aligned}
			\left\|e^{i t(-\Delta_\kappa)^{\frac{m}{2}}} f\right\|_{L_t^q L_{\kappa_1, x}^r L_{\kappa_2, y}^{\tilde{r}}}^2  
			& \lesssim \sum_{j \in \mathbb{Z}}\left\|e^{i t(-\Delta_\kappa)^{\frac{m}{2}}} \mathcal{P}_j  f\right\|_{L_t^q L_{\kappa_1, x}^r L_{\kappa_2, y}^{\tilde{r}}}^2 \\
			&= \sum_{j \in \mathbb{Z}}\left\|e^{i t(-\Delta_\kappa)^{\frac{m}{2}}} \mathcal{P}_j  \widetilde{P}_j f\right\|_{L_t^q L_{\kappa_1, x}^r L_{\kappa_2, y}^{\tilde{\tau}}}^2\\
			&   \lesssim \sum_{j \in \mathbb{Z}} 2^{2 j s}\left\|\widetilde{P}_j f\right\|_{L_{\kappa}^2}^2 \lesssim\|f\|_{\dot{H}_{\kappa}^s}^2. 
		\end{aligned}
		$$ This completes the proof of the theorem. 
\end{proof}
 Now we consider the particular case $m=1$, that is ,  Strichartz estimates for the Dunkl-wave equation. 
	\begin{theorem}\label{D Main Dunkl wave}
		Let $N_1:=d+2\gamma_1 \geq 2,$ $1 \leq k \leq N_1-1,$ $ 2 \leq \widetilde{r} \leq r<\infty $ and $ 2<q \leq \infty$. If $ q, r, \tilde{r}$ satisfy  
		\begin{align}\label{eq00W}
			\frac{2}{q} \leq(d+2\gamma_1-k-1)\left(\frac{1}{2}-\frac{1}{r}\right)+(k+2\gamma_2)\left(\frac{1}{2}-\frac{1}{\widetilde{r}}\right). 
		\end{align}
		Then we have
		\begin{align}\label{Deq11W}
			\left\|e^{i t(-\Delta_\kappa)^{\frac{1}{2}}} f\right\|_{L_t^q\left(\mathbb{R} ; L_{\kappa_1, x}^r\left(\mathbb{R}^{d-k} ; L_{\kappa_2,y}^{\tilde{r}}\left(\mathbb{R}^k\right)\right)\right)} \lesssim\|f\|_{\dot{H}_{\kappa}^s}
		\end{align}
		under the scaling condition
		\begin{align}\label{Deq2W}
			\frac{{{1}}}{q}=-s +(d+2\gamma_1-k)\left(\frac{1}{2}-\frac{1}{r}\right)+(k+2\gamma_2)\left(\frac{1}{2}-\frac{1}{\widetilde{r}}\right).
		\end{align}
	\end{theorem}   
	\begin{proof}
			The proof proceeds analogously to that of Theorem \ref{D Main Dunkl} with appropriate modification in Lemma \ref{Deq6}. In this case, however, Corollary \ref{fractional} modifies the decay estimate \eqref{oscillatory}, resulting in a decay rate of $\tfrac{N-1}{2}$.
	\end{proof}

	\section*{Acknowledgement}
	SSM is supported by the DST-INSPIRE Faculty Fellowship DST/INSPIRE/04/2023/002038.   IS is supported by the Institute post doctoral fellowship, Indian Institute of Technology Bombay.    VK is supported by the FWO Odysseus 1 grant G.0H94.18N: Analysis and Partial Differential Equations and the Methusalem program of the Ghent University Special Research Fund (BOF) (Grant number 01M01021). VK is also supported by FWO Senior Research Grant G011522N.

\end{document}